\documentclass[a4paper,12pt]{article}
\usepackage{graphicx}
\usepackage{amssymb}
\usepackage{epstopdf}
\usepackage{tikz-cd}
\usepackage{setspace}
\usepackage[all]{xy}
\usepackage{eqnarray}
\usepackage{amsmath,arydshln} 
\usepackage{tikz}
\usepackage[left=3cm,right=3cm,top=3cm,bottom=3cm]{geometry}
\usepackage[cyr]{aeguill}
\usepackage{url}
\usepackage[francais,english]{babel}
\usepackage[utf8]{inputenc}  
\usepackage[T1]{fontenc} 

\usepackage{hyperref}

\usepackage{amssymb}

\usepackage{amsthm}       

\theoremstyle{definition}

\newtheorem{fact}{Fait}
\theoremstyle{plain}

\newtheorem{thm}{Théorème}
\newtheorem{lem}{Lemme}

\newtheorem{prop}{Proposition}

\def\Legendre(#1,#2){%
\begin{pmatrix}
#1\cr 
\hdashline[1pt/1pt]
#2\cr
\end{pmatrix}}

\title{Sur le $p$-rang du groupe des classes de $\mathbf{Q}(N^{\frac{1}{p}})$}
\author{Emmanuel Lecouturier}
\date {08/02/16}

\begin{document}
\maketitle

\selectlanguage{francais}
\begin{abstract} 
Soient $N$ et $p$ deux nombres premiers $\geq 5$ tels que $p$ divise $N-1$. Nous estimons le $p$-rang du groupe des classes de $\mathbf{Q}(N^{\frac{1}{p}})$, en terme du comportement de certaines unités, dont nous considérons des logarithmes à valeur dans $\mathbf{F}_p$. Nous réinterprétons ces logarithmes grâce à la formule de Gross--Koblitz et à des identités sur la fonction Gamma $N$-adique. Un cas particulier (dont nous n'avons pas trouvé de preuve élémentaire) de nos formules s'énonce ainsi : supposons qu'il existe $a$, $b$ $\in \mathbf{Z}$ tels que $N = \frac{a^p+b^p}{a+b}$. Alors $(a+b) \cdot \left( \prod_{k=1}^{\frac{N-1}{2}} k^{8k} \right)$ est une puissance $p$-ième modulo $N$. De plus nous donnons une nouvelle démonstration sans utiliser le formes modulaires d'un résultat de Calegari et Emerton.
\end{abstract}

\selectlanguage{english}
\begin{abstract} 
Let $N$ and $p$ be two prime numbers $\geq 5$ such that $p$ divides $N-1$. We estimate the $p$-rank of the class group of $\mathbf{Q}(N^{\frac{1}{p}})$ in terms of the discrete logarithm, with values un $\mathbf{F}_p$, of certain units. Using the Gross--Koblitz formula and identities on the $N$-adic Gamma function, we explicitly compute these logarithms. A special case (for which we don't have an elementary proof) of our formula is the following: assume there are some integers $a$, $b$ such that $N = \frac{a^p+b^p}{a+b}$. Then $(a+b) \cdot \left( \prod_{k=1}^{\frac{N-1}{2}} k^{8k} \right)$ is a $p$-th power modulo $N$. Furthermore we give a new proof which doesn't use modular forms of a result of Calegari and Emerton.
\end{abstract}

\selectlanguage{francais}
\tableofcontents
 \
\section{Introduction}

Soient $N$ et $p$ deux nombres premiers $\geq 5$. Supposons que $N \equiv 1  \text{ (mod } p\text{)}$. Soit $\nu$ la valuation $p$-adique de $N-1$. Fixons un morphisme surjectif de groupes $ \log : (\mathbf{Z}/N\mathbf{Z})^{\times} \rightarrow \mathbf{Z}/p^{\nu}\mathbf{Z}$. Il nous arrivera de considérer la réduction modulo $p$ de $ \log$, c'est pourquoi nous préciserons toujours dans nos formules si nous travaillons seulement modulo $p$ ou bien modulo $p^{\nu}$. Fixons également une clôture algébrique $\overline{\mathbf{Q}}$ de $\mathbf{Q}$. 

Si $L$ est un corps de nombres, soit $\mathcal{C}_L$ le $p$-sous-groupe de Sylow du groupe des classes de $L$. Soit $r_L$ le $p$-rang de $\mathcal{C}_L$. 
Soit $\alpha \in \overline{\mathbf{Q}}$ tel que $\alpha^p=N$ et posons $K = \mathbf{Q}(\alpha)$. Le but de cet article est d'étudier l'entier $r_K$. Cet entier été étudié, notamment par Iimura (\cite{Iimura}) et plus récemment par Calegari et Emerton (\cite{CE}).
On peut montrer, via la théorie du genre, que $r_K \geq 1$ (\textit{i.e.} $\mathcal{C}_K \neq 0$).
En utilisant la formule des genres de Jaulent, Iimura parvient à estimer précisément $r_K$ en fonction de symboles de Hilbert, qui ont le désavantage de ne pas être très explicites. D'un autre côté, Calegari et Emerton donnent le critère explicite suivant :

\begin{thm}\label{CE}\cite[Theorem $1.3 \text{ }(ii)$]{CE}
Supposons que $\sum_{k=1}^{\frac{N-1}{2}} k\cdot  \log (k)  \equiv 0  \text{ (mod } p\text{)}$. Alors on a $r_K\geq 2$.
\end{thm}

Soient $\zeta_p$ et $\zeta_N$ dans $\overline{\mathbf{Q}}$ des racines primitives $p$-ièmes et $N$-ièmes de l'unité respectivement. La quantité $\sum_{k=1}^{\frac{N-1}{2}} k\cdot  \log (k)$ intervient naturellement dans l'étude du corps cyclotomique $\mathbf{Q}(\zeta_p, \zeta_N)$.

Soit $\omega : (\mathbf{Z}/p\mathbf{Z})^{\times} \rightarrow \mathbf{Z}_p^{\times}$ le caractère de Teichmüller (vérifiant $\omega(a) \equiv a \text{ (mod } p\text{)}$). Dans cet article, $\chi$ désigne un caractère $(\mathbf{Z}/p\mathbf{Z})^{\times} \rightarrow \mathbf{Z}_p^{\times}$. On notera aussi $\chi_0 : (\mathbf{Z}/p\mathbf{Z})^{\times} \rightarrow \mathbf{Z}_p^{\times}$ le caractère trivial. Si $G$ est un groupe fini, on note $\mathbf{Z}_p[G]$ l'algèbre du groupe $G$ et, si $g \in G$, on note $[g] \in \mathbf{Z}_p[G]$ l'élément correspondant. On note $I_G \subset \mathbf{Z}_p[G]$ l'idéal d'augmentation. On note $e_{\chi} = \frac{1}{p-1}\cdot \sum_{a \in \mathbf{F}_p^{\times}} \chi^{-1}(a)\cdot [a] \in \mathbf{Z}_p[(\mathbf{Z}/p\mathbf{Z})^{\times}]$ l'idempotent associé à $\chi$. Si $M$ est un groupe abélien, on note $M(\chi) = e_{\chi}\cdot (M \otimes_{\mathbf{Z}} \mathbf{Z}_p)$.

Soit $\zeta_N^{+}$ un générateur de l'unique sous-extension de degré $p$ de $\mathbf{Q}(\zeta_N)$, et soit $K_0 = \mathbf{Q}(\zeta_p, \zeta_N^+)$. Si $\chi : (\mathbf{Z}/p\mathbf{Z})^{\times} \rightarrow \mathbf{Z}_p^{\times}$ est un caractère, notons $r_{K_0}(\chi)$ le $p$-rang de $\mathcal{C}_{K_0}(\chi)$.

\begin{thm}\label{r_0}
On a $r_{K_0}(\omega^{-1}) \geq 1$. 
On a  $r_{K_0}(\omega^{-1})\geq 2$ si et seulement si  $\sum_{k=1}^{\frac{N-1}{2}} k\cdot \log(k) \equiv 0  \text{ (mod } p\text{)}$.
\end{thm}

Le sens direct est implicitement dû à Schoof \cite{Schoof} (comme nous le verrons dans la section $4$), qui semble donc être le premier à considérer la quantité $\sum_{k=1}^{\frac{N-1}{2}} k\cdot  \log(k)  \text{ (mod } p\text{)}$. Dans \cite{CE}, Calegari et Emerton construisent par la théorie des déformations un anneau $R$ qu'ils identifient au complété $\mathfrak{p}$-adique $\mathbf{T}_{\mathfrak{p}}$ de l'algèbre de Hecke $\mathbf{T}$ considérée par Mazur dans \cite{Mazur_Eisenstein}, où $\mathfrak{p}$ est le $p$-idéal maximal d'Eisenstein. L'étude de ces anneaux par la théorie du corps de classe d'une part et les symboles modulaires d'autre part (\cite[Théorème 2]{Merel}) mène au théorème~\ref{CE}. Cela permet de montrer le sens réciproque du théorème~\ref{r_0} (bien que ce fait n'ait pas été noté par \cite{CE}). Nous donnons cependant une preuve directe, par la pure théorie algébrique des nombres, des théorèmes~\ref{CE} et~\ref{r_0}, ce qui répond à une question de Calegari et Emerton \cite[note p.$99$]{CE}.

Comme le note Schoof dans l'introduction de \cite{Schoof}, ce résultat qui concerne la structure d'un groupe des classes cyclotomique n'est pas prédit par la théorie d'Iwasawa qui ne considère que le cardinal des groupes de classes.

Schoof établit (implicitement) le résultat suivant dans \cite{Schoof} : 
\begin{thm}\label{r_0>i}
Soit $i$ un entier tel que $1\leq i \leq p-1$.
Si $r_{K_0}(\omega^{-1})>i$, alors pour tout $1 \leq j \leq i$, on a $\sum_{k=1}^{\frac{N-1}{2}} k\cdot \log(k)^j \equiv 0  \text{ (mod } p\text{)}$. 
\end{thm}
La réciproque du théorème~\ref{r_0>i} est vraie au moins pour $i \leq 2$. Nous le montrerons en effet dans un prochain article grâce aux symboles modulaires et au travail de Calegari--Emerton.

La réciproque du théorème~\ref{CE} est fausse en général (contrairement à ce qu'espéraient Calegari et Emerton sur la base de résultats numériques, \textit{cf.} \cite{CE} dernier paragraphe de la page $142$) : par exemple si $p=7$ et $N = 337$, on a $\sum_{k=1}^{\frac{N-1}{2}} k \cdot  \log(k) \neq 0  \text{ (mod } p\text{)}$ mais $r_K=2$ (en fait le groupe des classes de $K$ est isomorphe à $(\mathbf{Z}/7\mathbf{Z})^2$ par SAGE en supposant GRH pour les fonctions L de Dirichlet). Pour montrer leur critère, Calegari et Emerton construisent explicitement (en utilisant les représentations galoisiennes) deux $\mathbf{F}_p$-extensions abéliennes indépendantes partout non ramifiées de $K$. Cependant les résultats de Iimura montrent que ces deux extensions ne correspondent qu'aux deux premiers crans d'une filtration de $\mathcal{C}_K$. Iimura, en étudiant ces deux premiers crans, s'intéresse au cas particulier suivant :
\begin{thm} [\cite{Iimura}] \label{Iimura}
Supposons qu'il existe deux entiers $a$ et $b$ tels que $N = \frac{a^p+b^p}{a+b}$. Si $ \log(a+b) \equiv 0  \text{ (mod } p\text{)}$, alors on a $r_K \geq 2$.
\end{thm}
On peut soupçonner que $ \log(a+b)$ et $\sum_{k=1}^{\frac{N-1}{2}} k\cdot  \log(k)$ sont proportionnels. C'est effectivement le cas et on peut faire le lien entre les théorèmes~\ref{CE} et~\ref{Iimura} : 
\begin{thm}\label{a+b}
Supposons que $N = \frac{a^p+b^p}{a+b}$ avec $a$, $b$ $\in \mathbf{Z}$. Alors on a 
$$ \log(a+b) \equiv -8\cdot \sum_{k=1}^{\frac{N-1}{2}} k\cdot  \log(k)  \text{ (mod } p\text{)} \text{ .}$$
\end{thm}
Cette identité, dont nous n'avons pas trouvé de preuve élémentaire, est un cas particulier d'un résultat plus général que nous décrivons maintenant.
Supposons que $N$ est une norme dans $\mathbf{Q}(\zeta_p)$, \textit{i.e.} qu'il existe $u \in \mathbf{Z}[\zeta_p]$ tel que $N = \text{N}_{\mathbf{Q}(\zeta_p)/\mathbf{Q}}(u)$.  

Soit $\chi : (\mathbf{Z}/p\mathbf{Z})^{\times} \rightarrow \mathbf{Z}_p^{\times}$ un caractère impair ($\chi(-1)=-1$), différent de $\omega$, et tel que $B_{1,\chi^{-1}} := \sum_{a=1}^{p-1} a\cdot \chi^{-1}(a)$ soit dans $\mathbf{Z}_p^{\times}$ (autrement dit, si on écrit $\chi = \omega^{-i}$ avec $1 \leq i \leq p-4$, alors $p$ ne divise pas le nombre de Bernoulli $B_{i+1}$ car par \cite[Corollary $5.15$]{Washington} on a $B_{1, \omega^i} \equiv \frac{B_{i+1}}{i+1}  \text{ (mod } p \text{)}$). 

Si $a \in \mathbf{F}_p^{\times}$, soit $\sigma_a \in \text{Gal}(\mathbf{Q}(\zeta_p, \alpha)/\mathbf{Q}(\alpha))$ tel que $\sigma_a(\zeta_p) = \zeta_p^a$ ($\sigma_a$ s'identifie à un élément de $\text{Gal}(\mathbf{Q}(\zeta_p)/\mathbf{Q})$ par restriction).

Comme $N \equiv 1 \text{ (mod } p\text{)}$, l'idéal $(N)$ est totalement décomposé dans $\mathbf{Q}(\zeta_p)$. Les idéaux premiers au dessus de $N$ sont les $(\sigma_a(u))$ pour $a \in \mathbf{F}_p^{\times}$. Si $a \not\equiv 1  \text{ (mod } p\text{)}  $, notons que $\sigma_a(u)$ est premier à $(u)$. De plus $\mathbf{Z}[\zeta_p]/(u)$ s'identifie canoniquement à $\mathbf{F}_N^{\times}$. On peut donc considérer $ \log(\sigma_a(u))$. 
\begin{thm}\label{P}
On a : 
$$\sum_{a=2}^{p-1} (\chi^{-1}(a)-1)\cdot  \log(\sigma_a(u)) \equiv \frac{1}{B_{1,\chi^{-1}}} \cdot \left( \sum_{k=1}^{N-1} \left( \sum_{a=1}^{k-1} \chi^{-1}(a)\right)\cdot  \log(k)\right)  \text{ (mod } p^{\nu}\text{)}$$
\end{thm}

Notons que la quantité de gauche a un sens pour tout caractère $\chi : (\mathbf{Z}/p\mathbf{Z})^{\times} \rightarrow \mathbf{Z}_p^{\times}$ (y compris pour $\chi$ pair). Nous n'avons pas pu trouver de formule explicite si $\chi$ ne vérifie pas les conditions ci-dessus. On a une formule analogue sans supposer que les idéaux au-dessus de $N$ dans $\mathbf{Q}(\zeta_p)$ sont principaux, en utilisant le fait que  $\mathcal{C}_{\mathbf{Q}(\zeta_p)}(\chi)=0$ par le théorème d'Herbrand--Ribet (\textit{cf.} Théorème~\ref{identite_u}).
 
 Les calculs figurant dans \cite[p. $164-165$]{Iimura} permettent de montrer que les membres de gauche des théorèmes~\ref{a+b} et~\ref{P} sont congrus modulo $p$ (lorsque $N = \frac{a^p+b^p}{a+b}$, $u=a+b\cdot \zeta_p$ et $\chi = \omega^{-1}$), si bien que le théorème~\ref{a+b} résulte du théorème~\ref{P} (en utilisant le fait que $B_{1,\omega^{-1}} \equiv \frac{B_2}{2} \equiv \frac{1}{12} \text{ (mod }p\text{)}$).

L'identité du théorème~\ref{P} provient essentiellement d'une conjecture (tordue par $\chi$) de Gross (\cite[Conjecture $7.6$]{Gross_refined}) : le terme de gauche correspond à un régulateur des $N$-unités de $\mathbf{Q}(\zeta_p)$ et le terme de droite correspond à l'image d'un élément de Stickelberger dans $I/I^2$ où $I$ est l'idéal d'augmentation dans $\mathbf{Z}_p[(\mathbf{Z}/N\mathbf{Z})^{\times}]$ (\textit{cf.} la proposition~\ref{Stickelberger}). Gross relie une telle quantité à la dérivée d'une fonction $L$ abélienne en $s=0$. Cependant nous n'avons pas trouvé dans la littérature de preuve d'une version de la conjecture de Gross sur $\mathbf{Q}$ donnant lieu à cette identité. Nous avons donc prouvé directement l'identité du théorème~\ref{P}, en suivant l'idée de la preuve de la conjecture de Gross usuelle (\textit{cf.} \cite{Hayes}), c'est à dire en utilisant la formule de Gross--Koblitz (\textit{cf.} \cite[Theorem $1.7$]{Gross_Koblitz}).

Si $i \in \{1,2,...,p-2\}$, posons :
\begin{align*}
\mathcal{S}_i & \equiv \sum_{k=1}^{N-1} \left(\sum_{a=1}^{k-1} a^i \right) \cdot  \log(k)  \text{ (mod } p\text{)} \\& \equiv  \frac{1}{i+1}\cdot \left( \sum_{k=1}^{N-1} B_{i+1}(k) \cdot  \log(k)  \right) \text{ (mod } p\text{)}
\end{align*}
où $B_{i+1}(X)$ est le $i+1$-ième polynôme de Bernoulli (et $B_{i+1} = B_{i+1}(0)$ est le $i+1$-ième nombre de Bernoulli).
Soit $\mu$ le nombre de $i$ dans $\{1,3,...,p-4\}$ impairs tels que $B_{i+1} \not\equiv 0  \text{ (mod } p\text{)}$ et $\mathcal{S}_i \not\equiv 0 \text{ (mod } p\text{)}$.

On verra (\textit{cf.} lemme~\ref{B_2}) que $\mathcal{S}_1 = \frac{-2}{3} \cdot \sum_{k=1}^{\frac{N-1}{2}} k\cdot  \log(k)  \text{ (mod } p\text{)}$. On pose $\mu_1=1$ si  $\mathcal{S}_1 \equiv 0 \text{ (mod } p\text{)}$ et $\mu_1=0$ sinon.

\begin{thm}\label{class_group}
On a :
$$ 1+ \mu_1 \leq r_K \leq r_{\mathbf{Q}(\zeta_p)} + p-2 - \mu - \emph{Card}\left( \mathbf{Z}[\zeta_p]^{\times}/\emph{N}_{\mathbf{Q}(\zeta_p, N^{\frac{1}{p}})/\mathbf{Q}(\zeta_p)}(\mathbf{Q}(\zeta_p, N^{\frac{1}{p}})^{\times}) \cap \mathbf{Z}[\zeta_p]^{\times} \right)  \text{ .}$$
\end{thm}

En particulier on retrouve le théorème~\ref{CE}. Comme nous le verrons dans la preuve, on a une description plus fine mais plus compliquée de $r_K$ et $r_{K_0}(\omega^{-1})$, et nous expliquerons quelles sont les difficultés pour améliorer les bornes explicitement.

Pour trouver des couples $(p,N)$ avec $r_K\geq 2$ mais $\sum_{k=1}^{\frac{N-1}{2}} k\cdot  \log(k) \not\equiv 0  \text{ (mod } p\text{)}$, il est judicieux au vu du théorème~\ref{class_group} de chercher l'annulation des sommes $S_i$ pour $i$ dans $\{1,3,...,p-4\}$ impair tel que $B_{i+1} \not\equiv 0  \text{ (mod } p\text{)}$. Par exemple pour $p=7$ et $N=337, 631, 659, 1303, 1723$, on a $S_1 \neq 0$ mais $S_3 = 0$, et dans ces cas on vérifie en utilisant SAGE (sous GRH) que $r_K=2$. 

Pour $p=5$ nous conjecturons, comme Calegari--Emerton, que la réciproque du théorème~\ref{CE} est vraie. 

\begin{thm}\label{p=5}
Supposons que $p=5$ (dans ce cas $\mathbf{Q}(\zeta_p)$ est principal et $1+\zeta_p$ est une unité fondamentale de $\mathbf{Z}[\zeta_p]$). Soit $u \in \mathbf{Z}[\zeta_p]$ tel que $(u)$ est un idéal premier au-dessus de $N$ dans $\mathbf{Q}(\zeta_p)$.
On a 
$$ 1\leq r_K \leq 3 \text{ .}$$
(i) Si $\sum_{k=1}^{\frac{N-1}{2}} k\cdot log(k) \equiv 0  \text{ (mod } p \text{)} $, alors $r_K \geq 2$, et dans ce cas, si $r_K \geq 3$, alors 
$$\sum_{a=1}^{4} a^2\cdot \emph{log}(1+\zeta_p^a) \equiv 0 \text{ (mod } p \text{)} $$
et
$$\sum_{a \in \mathbf{F}_p^{\times} \text{, } a \not\equiv 1} (a^2-1)\cdot \emph{log}(\sigma_a(u)) \equiv 0 \text{ (mod } p \text{)} $$
(la deuxième quantité est indépendante du choix de $u$ sous l'hypothèse que la première quantité est nulle) \\
(ii) Si $\sum_{k=1}^{\frac{N-1}{2}} k\cdot log(k) \not\equiv 0 \text{ (mod } p \text{)}$ et que $r_K \geq 2$ alors 
$$\sum_{a=1}^{4} a^2\cdot \emph{log}(1+\zeta_p^a) \equiv 0 \text{ (mod } p \text{)} $$
et
$$\sum_{a \in \mathbf{F}_p^{\times} \text{, } a \not\equiv 1} (a^2-1)\cdot \emph{log}(\sigma_a(u)) \equiv 0 \text{ (mod } p \text{)} \text{ .}$$

\end{thm}

Nous conjecturons que le cas $(ii)$ ($r_K>1$ et $\sum_{k=1}^{\frac{N-1}{2}} k\cdot log(k) \not\equiv 0 \text{ (mod } p \text{)}$) ne se produit jamais. Nous conjecturons que dans le cas $(i)$ où $\sum_{k=1}^{\frac{N-1}{2}} k\cdot log(k) \equiv 0  \text{ (mod } p \text{)} $, le fait que $r_K = 3$ est équivalent à  $$\sum_{a=1}^{4}a^2 \cdot \emph{log}(1+\zeta_p^a) \equiv 0 \text{ (mod } p \text{)} $$ 
et $$\sum_{a \in \mathbf{F}_p^{\times} \text{, } a \not\equiv 1} (a^2-1)\cdot \emph{log}(\sigma_a(u)) \equiv 0 \text{ (mod } p \text{)} \text{ .}$$
Nous avons vérifié cette conjecture numériquement avec SAGE, sous GRH, pour $N < 5000$ (le plus petit exemple pour lequel $r_K=3$ est $N=211$). Il serait intéressant d'avoir une formule plus explicite pour ces deux sommes de logarithmes (comme dans le théorème~\ref{P}).

Pour $p=3$  et $N \equiv 1  \text{ (mod } 9 \text{)}$ , le groupe $\mathcal{C}_{K}$ a été étudié par les mêmes méthodes dans \cite{Gerth} et répond à une conjecture de Calegari et Emerton \cite[$p.141$]{CE}.

Pour $p=2$ et $N \equiv 9  \text{ (mod } 16 \text{)}$, Calegari et Emerton (\cite[Theorem $1.1$]{CE}) ont relié le cardinal de $\mathcal{C}_{\mathbf{Q}(\sqrt{-N})}$ au rang de l'algèbre de Hecke mentionnée ci-dessus, mais nous ne connaissons pas de bornes pour le $2$-rang du $2$-sylow de $\mathcal{C}_{\mathbf{Q}(\sqrt{-N})}$ (qui est $\neq 0$ par la théorie du genre de Gauss).

Dans la deuxième partie de cet article, nous reprenons les méthodes d'Iimura pour étudier l'entier $r_K$, et donnons des bornes en termes de symboles de Hilbert. Dans la troisième partie, nous explicitons ces symboles de Hilbert en termes du logarithme discret d'une certaine $N$-unité (qui est liée à une somme de Gauss). Dans la quatrième partie, nous étudions l'entier $r_{K_0}$ avec les même méthodes que précédemment, et nous faisons le lien avec \cite{Schoof}. Dans la cinquième partie, nous faisons le lien entre la fonction Gamma $N$-adique et certains éléments de Stickelberger. Dans la sixième partie, nous finissons la démonstration des théorèmes énoncés plus hauts grâce à la formule de Gross--Koblitz.

Je voudrais remercier Loïc Merel, mon directeur de thèse, pour avoir guidé ce travail et pour ses conseils de rédaction. Il a attiré mon attention sur la quantité $\sum_{k=1}^{\frac{N-1}{2}} k \cdot \log(k)$, et a remarqué l'analogie avec une valeur de dérivée de fonction $L$. Il a aussi remarqué que certaines identités de ce papier ressemblent à des formules de nombre de classe, ce qui m'a mené à la conjecture de Gross. Je souhaite aussi remercier Henri Cohen qui m'a souligné le lien entre les travaux de René Schoof sur les corps cyclotomiques et la quantité $\sum_{k=1}^{\frac{N-1}{2}} k \cdot \log(k)$. Enfin je tiens à remercier René Schoof pour m'avoir envoyé l'article en question. 
\newpage
\section{Rappels des résultats d'Iimura et Jaulent (et compléments)}
Cette partie est largement due à Iimura, mais pour le confort du lecteur nous reprenons ses méthodes (et ses notations autant que possible) dans notre contexte, sans constamment renvoyer à l'article original (\cite{Iimura}). 

Soit $K_1 = \mathbf{Q}(\zeta_p, \alpha)$ la clôture galoisienne de $K$ dans $\overline{\mathbf{Q}}$. Soient $G = Gal(K_1/\mathbf{Q})$, $S=\text{Gal}(K_1/\mathbf{Q}(\zeta_p))$ et $T = \text{Gal}(K_1/K)$. On identifie $T$ à $\text{Gal}(\mathbf{Q}(\zeta_p)/\mathbf{Q})$ par restriction, ainsi qu'à $\mathbf{F}_p^{\times}$ via l'isomorphisme $\mathbf{F}_p^{\times} \rightarrow \text{Gal}(\mathbf{Q}(\zeta_p)/\mathbf{Q})$ donné par $a \mapsto \sigma_a$ (rappelons que $\sigma_a(\zeta_p) = \zeta_p^a$). Ainsi un caractère $\chi : (\mathbf{Z}/p\mathbf{Z})^{\times} \rightarrow \mathbf{Z}_p^{\times}$ s'identifie à un caractère de $T$. Fixons dans ce qui suit un générateur $\sigma$ de $S$. Rappelons que $\omega$ dénote le caractère de Teichmüller.

Le groupe $\mathcal{C}_{K_1}$ est un $\mathbf{Z}_p[G]$-module à gauche. 
Si $0 \leq i \leq p-2$ est un entier, on note $J_i = \left( I_S^i\cdot \mathcal{C}_{K_1}\right)\cdot \mathcal{C}_{K_1}^p$ (où $I_S$ est l'idéal d'augmentation de $\mathbf{Z}_p[S]$) et on pose $I_i = J_i/J_{i+1}$, qui est un $\mathbf{F}_p[G]$-module. Si $\chi : T \rightarrow \mathbf{Z}_p^{\times}$ est un caractère, notons $\alpha_i(\chi)$ le $\mathbf{F}_p$-rang de $e_{\chi}(I_i)$.

Comme le degré $[K_1 : K] = p-1$ est premier à $p$, la flèche naturelle $\mathcal{C}_{K} \rightarrow \mathcal{C}_{K_1}$ est injective. Rappelons que $\chi_0$ est le caractère trivial de $T$.
Le fait suivant est la clé pour l'étude de $r_K$.
\begin{prop}\label{decomposition}\cite[Lemma $1.1$]{Iimura}
On a $\mathcal{C}_{K} = \mathcal{C}_{K_1}(\chi_0)$. On a $r_K = \sum_{i=0}^{p-2} \alpha_i(\chi_0)$.
\end{prop}
On pose $$\delta = \sum_{i=0}^{p-1} [\sigma^i] \text{ .}$$

\begin{lem}\label{delta}\cite[Lemme $1$]{Jaulent}
On a $$p \in (([\sigma]-1)^{p-1} - \delta) \cdot  \mathbf{Z}_p[S]^{\times} \text{ .}$$
\end{lem}

Prouvons la proposition~\ref{decomposition}. 
On a clairement $\mathcal{C}_K \subset \mathcal{C}_{K_1}(\chi_0)$. Réciproquement, si une classe d'idéaux $\mathfrak{c}$ est fixée par $T$, on a $N_{K_1/K}(\mathfrak{c})=\mathfrak{c}^{p-1} \in \mathcal{C}_K$ donc $\mathfrak{c} \in \mathcal{C}_K$. Par le lemme~\ref{delta}, $\mathcal{C}_{K}^p=e_{\chi_0}\cdot(I_{S}^{p-1}\cdot \mathcal{C}_{K_1})$ (on utilise le fait que $\delta \cdot \mathcal{C}_K = N_{K/\mathbf{Q}} (\mathcal{C}_K) = 1$ car $\mathbf{Q}$ a un groupe des classes trivial).

Il s'agit donc de comprendre les $\alpha_i(\chi_0)$. 

Notons que si $\tau \in T$, alors 
$$\tau \sigma \tau^{-1} = \sigma^{\omega(\tau)} \text{ .}$$
On pose, en suivant par exemple Jaulent dans \cite{Jaulent}, $$\Delta = \frac{1}{p-1}\cdot \left(\sum_{\tau \in T} \omega(\tau^{-1})\cdot [\sigma^{\omega(\tau)}]\right) \in \mathbf{Z}_p[S] \text{ .}$$

\begin{fact}\label{shift}\cite[Proposition $2$]{Jaulent}
L'élément $\Delta$ de l'algèbre de groupe $\mathbf{Z}_p[T]$ est un générateur de l'idéal d'augmentation. De plus on a la relation suivante, pour tout caractère $\psi : T \rightarrow \mathbf{Z}_p^{\times}$ :
$$ e_{\psi\cdot \omega} \cdot \Delta = \Delta \cdot e_{\psi}$$
où on rappelle que $e_{\psi} = \frac{1}{p-1}\cdot \left(\sum_{\tau \in T} \psi(\tau^{-1})\cdot [\tau] \right) \in \mathbf{Z}_p[T]$.
\end{fact}

En suivant Iimura, on a la suite exacte suivante, pour tout $0 \leq i \leq p-3$ :
$$1 \rightarrow U_i \rightarrow I_i \rightarrow I_{i+1} \rightarrow 1$$
où la flèche $I_i \rightarrow I_{i+1}$ est la multiplication par $\Delta$.
Par le fait~\ref{shift}, cette suite induit pour tout caractère $\chi$ de $T$ une suite exacte :
\begin{equation}
1 \rightarrow U_i(\chi \cdot \omega^{-1}) \rightarrow I_i(\chi\cdot \omega^{-1}) \rightarrow I_{i+1}(\chi)\rightarrow 1\text{ .}
\label{suite_exacte_I}
\end{equation}
Il semble difficile de comprendre les $U_i$ en général, mais $U_0$ est calculable en termes de symboles de Hilbert. La théorie du genre nous donne par ailleurs une description de $I_0(\chi)$ pour tout $\chi$. La suite exacte permet alors une majoration de $I_i(\chi)$. 

Pour énoncer la formule des genres de Jaulent, on a besoin de quelques notations. Soit $\chi : T \rightarrow \mathbf{Z}_p^{\times}$ un caractère. Soit $n(\chi)$ le $p$-rang du groupe $(\mathbf{Z}[\zeta_p]^{\times}/\mathbf{Z}[\zeta_p]^{\times}\cap N_{K_1/\mathbf{Q}(\zeta_p)}(K_1))(\chi)$. Soit $\overline{\mathcal{C}}_{K_1}$ le noyau de $N_{K_1/\mathbf{Q}(\zeta_p)} : \mathcal{C}_{K_1} \rightarrow \mathcal{C}_{\mathbf{Q}(\zeta_p)}$. On a $I_{S} \cdot \mathcal{C}_{K_1} \subset \overline{\mathcal{C}}_{K_1}$ car $\delta\cdot ([\sigma]-1)=0$. Soit $J = \overline{\mathcal{C}}_{K_1}/(I_{S} \cdot\mathcal{C}_{K_1})$, et $j(\chi)$ le $\chi$-rang de $J$. Remarquons que $J^p =1$ par le lemme~\ref{delta}.

La proposition suivante, qui utilise la théorie du genre, sera importante pour comprendre $\alpha_0(\chi)$ (\textit{cf.} la proposition~\ref{sansnom1} ci-dessous).
\begin{prop}\label{genus}
On a, pour tout caractère $\chi : T \rightarrow \mathbf{Z}_p^{\times}$ :
$$j(\chi) = 1-n(\chi) $$
sauf si $N \equiv 1  \text{ (mod } p^2\text{)}$ et que $\chi=\omega$, auquel cas on a 
$$j(\omega) = 0 \text{ .} $$
\end{prop}
\begin{proof}
D'après \cite[Proposition $8$]{Jaulent}, on a une suite exacte :
\begin{equation}
1 \rightarrow \mathbf{Z}[\zeta_p]^{\times}/\mathbf{Z}[\zeta_p]^{\times} \cap N_{K_1/\mathbf{Q}(\zeta_p)}(K_1) \rightarrow \oplus_{\mathfrak{p}} I_{\mathfrak{p}} \rightarrow Ker( Gal(\tilde{K_1}/\mathbf{Q}(\zeta_p)) \overset{res}{\longrightarrow}Gal(\tilde{\mathbf{Q}(\zeta_p)}/\mathbf{Q}(\zeta_p))) \rightarrow 1
\label{suite_Jaulent}
\end{equation}
Ici $\tilde{K}_1$ est la plus grande extension d'ordre une puissance de $p$ non ramifiée sur $K_1$ dans $\overline{\mathbf{Q}}$ qui est abélienne sur $\mathbf{Q}(\zeta_p)$. Le corps $\tilde{\mathbf{Q}(\zeta_p)}$ est le $p$-Hilbert de $\mathbf{Q}(\zeta_p)$ et $\oplus_{\mathfrak{p}} I_{\mathfrak{p}}$ est la somme directe des groupes d'inerties dans $Gal(\tilde{K_1}/\mathbf{Q}(\zeta_p))$ en les idéaux premiers $\mathfrak{p}$ de $\mathbf{Z}[\zeta_p]$ se ramifiant dans $K_1$. La flèche de gauche est donnée par le symbole d'Artin en les différents $\mathfrak{p}$, le point important est l'exactitude au milieu. Le groupe $\text{Ker}( Gal(\tilde{K_1}/\mathbf{Q}(\zeta_p)) \overset{res}{\longrightarrow}Gal(\tilde{\mathbf{Q}(\zeta_p)}/\mathbf{Q}(\zeta_p)))$ est une extension de $S$ par $J$. En effet on a un diagramme commutatif dont les lignes sont exactes : 
$$\xymatrix{
   0 \ar[r] &  \text{Gal}(\tilde{K_1}/K_1) \ar[r] \ar[d]_{\text{res}}  &  \text{Gal}(\tilde{K_1}/\mathbf{Q}(\zeta_p))  \ar[r] \ar[d]_{\text{res}}  &  S \ar[r] \ar[d]& 0  \\
    0 \ar[r] &  \text{Gal}(\tilde{\mathbf{Q}(\zeta_p)}/\mathbf{Q}(\zeta_p)) \ar[r]^{=}  & \text{Gal}(\tilde{\mathbf{Q}(\zeta_p)}/\mathbf{Q}(\zeta_p))  \ar[r] &  0 \ar[r] & 0   
  }$$
Le noyau de la flèche verticale de gauche est $J$ par la théorie globale du corps de classe (en effet, la restriction des groupes de Galois correspond à la norme des idéaux). Le lemme du serpent nous donne alors une suite exacte :
$$ 0 \rightarrow J \rightarrow \text{Ker}( Gal(\tilde{K_1}/\mathbf{Q}(\zeta_p)) \overset{res}{\longrightarrow}Gal(\tilde{\mathbf{Q}(\zeta_p)}/\mathbf{Q}(\zeta_p))) \rightarrow S \rightarrow 0 \text{ .}$$
 Comme les $I_{\mathfrak{p}}$ sont d'ordre $p$, la suite exacte~\eqref{suite_Jaulent} est une suite exacte de de $\mathbf{F}_p$-espace vectoriels, qui est $T$-équivariante (pour l'action usuelle de $T$ sur $\oplus_{\mathfrak{p}} I_{\mathfrak{p}})$. 

Considérons maintenant la $\chi$-partie de cette suite exacte.

Si $\chi \neq \omega$, on a $S(\chi)=0$. Seuls les idéaux premiers au-dessus de $N$ dans $\mathbf{Z}[\zeta_p]$ interviennent dans le terme $(\oplus I_{\mathfrak{p}})(\chi)$, et comme $T$ agit transitivement sur ces idéaux, le théorème~\ref{genus} est prouvé.

Si $\chi = \omega$, on a $S(\chi) \simeq \mathbf{F}_p$. Les idéaux premiers $\mathfrak{p}$ de $\mathbf{Z}[\zeta_p]$ intervenant dans $\mathbf{Z}[\zeta_p]$ sont ceux au-dessus de $N$ et de $p$ sauf si $N \equiv 1  \text{ (mod } p^2\text{)}$ auquel cas $K_1$ n'est pas ramifié en $p$ au-dessus de $\mathbf{Q}(\zeta_p)$ et seuls les idéaux premiers $\mathfrak{p}$ au-dessus de $N$ interviennent. Cela achève la preuve du théorème~\ref{genus} dans ce cas.
\end{proof}

En fait on montrera dans le lemme~\ref{j(w)} qu'on a toujours $j(\omega)=0$.

Comme $K_1$ est ramifié en $N$ au-dessus de $\mathbf{Q}(\zeta_p)$, la flèche naturelle $\mathcal{C}_{\mathbf{Q}(\zeta_p)} \rightarrow \mathcal{C}_{K_1}$ est injective. On voit donc $\mathcal{C}_{\mathbf{Q}(\zeta_p)}$ comme un sous-groupe de $\mathcal{C}_{K_1}$. Par ailleurs la théorie globale du corps de classe nous donne :
\begin{equation}
\text{N}_{K_1/\mathbf{Q}(\zeta_p)}(\mathcal{C}_{K_1}) = \mathcal{C}_{\mathbf{Q}(\zeta_p)}
\label{norme}
\end{equation}
(en effet, en termes d'extensions abéliennes, la norme correspond à la restriction des groupes de Galois, et la restriction est ici surjective car $K_1$ est totalement ramifié en $N$ sur $\mathbf{Q}(\zeta_p)$).

\begin{lem}\label{sans_nom_2}
On a dans $\mathcal{C}_{K_1}$ :
$$(I_{S}\cdot\mathcal{C}_{K_1} )\cdot \mathcal{C}_{K_1}^p = (I_{S}\cdot\mathcal{C}_{K_1} )\cdot \mathcal{C}_{\mathbf{Q}(\zeta_p)} $$
(où l'on voit abusivement $\mathcal{C}_{\mathbf{Q}(\zeta_p)}$ comme un sous-groupe de $\mathcal{C}_{K_1}$).
\end{lem}
\begin{proof}
En effet, par le lemme~\ref{delta}, $$\mathcal{C}_{K_1}^p = (I_{S}^{p-1}\cdot\mathcal{C}_{K_1} )\cdot (\delta \cdot \mathcal{C}_{K_1}) \subset (I_{S}\cdot\mathcal{C}_{K_1} )\cdot \text{N}_{K_1/\mathbf{Q}(\zeta_p)}(\mathcal{C}_{K_1}) = (I_{S}\cdot\mathcal{C}_{K_1} )\cdot \mathcal{C}_{\mathbf{Q}(\zeta_p)}$$ la dernière égalité venant de la remarque précédente. Réciproquement,
$$\mathcal{C}_{\mathbf{Q}(\zeta_p)} = \text{N}_{K_1/\mathbf{Q}(\zeta_p)}(\mathcal{C}_{K_1}) = \delta \cdot \mathcal{C}_{K_1} = (I_{S}^{p-1}\cdot\mathcal{C}_{K_1} )\cdot (\mathcal{C}_{K_1})^p \subset  (I_{S}\cdot\mathcal{C}_{K_1} )\cdot \mathcal{C}_{K_1}^p$$
la dernière égalité venant du lemme~\ref{delta}. 
\end{proof}

Soit $\psi : \overline{\mathcal{C}}_{K_1} \rightarrow \overline{\mathcal{C}}_{K_1}/(I_{S}\cdot  \mathcal{C}_{K_1}) = J$ la projection canonique.

Notons $r_{\mathbf{Q}(\zeta_p)}(\chi)$ le $p$-rang de $e_{\chi}(\mathcal{C}_{\mathbf{Q}(\zeta_p)})$ et $\overline{\mathcal{C}}_{\mathbf{Q}(\zeta_p)} = \{c \in \mathcal{C}_{\mathbf{Q}(\zeta_p)}, c^p=1\}$.
La proposition suivante est la clé pour comprendre $\alpha_0(\chi)$.
\begin{prop}\label{sansnom1}\cite[Lemma $1.3$]{Iimura}
Pour tout caractère $\chi : T \rightarrow \mathbf{Z}_p^{\times}$, on a :
$$\alpha_0(\chi) = r_{\mathbf{Q}(\zeta_p)}(\chi) + j(\chi)- \text{dim}_{\mathbf{F}_p}(\psi(e_{\chi}\cdot \overline{\mathcal{C}}_{\mathbf{Q}(\zeta_p)})) \leq r_{\mathbf{Q}(\zeta_p)}(\chi) + j(\chi) \text{ .}$$
\end{prop}
\begin{proof}
Posons $B = \mathcal{C}_{K_1}/(I_S \cdot \mathcal{C}_{K_1})$. Alors $\alpha_0(\chi) = \text{dim}_{\mathbf{F}_p}((e_{\chi}\cdot B)/(e_{\chi}\cdot B)^p)$.
Considérons le diagramme commutatif dont les lignes sont exactes et les flèches verticales sont des inclusions : 
$$\xymatrix{
   1 \ar[r] & \psi(e_{\chi}\cdot (\mathcal{C}_{K_1}^p \cap \overline{\mathcal{C}}_{K_1})) \ar[r] \ar[d] &  B(\chi)^p  \ar[r] \ar[d]&  \mathcal{C}_{\mathbf{Q}(\zeta_p)}(\chi)^p \ar[r] \ar[d] & 1  \\
    1 \ar[r] & J(\chi) \ar[r] &  B(\chi)  \ar[r] &  \mathcal{C}_{\mathbf{Q}(\zeta_p)}(\chi) \ar[r] & 1
  }$$
Les troisièmes flèches horizontales correspondent à la norme. Par le lemme~\ref{sans_nom_2}, on a  $\psi(\mathcal{C}_{K_1}^p \cap \overline{\mathcal{C}}_{K_1}) = \psi(\overline{\mathcal{C}}_{\mathbf{Q}(\zeta_p)})$. Cela achève la démonstration de la proposition~\ref{sansnom1} par le lemme du serpent.
\end{proof}

En particulier, par la suite exacte~\eqref{suite_exacte_I} on a :
\begin{equation}
\alpha_i(\chi_0) \leq \alpha_0(\omega^{-i}) \leq 1-n(\omega^{-i})+r_{\mathbf{Q}(\zeta_p)}(\omega^{-i}) \text{ .}
\label{majoration_grossière}
\end{equation}
On va raffiner cette majoration pour certains $i$. 

On suppose désormais, sauf mention du contraire, que $i$ est impair, $i \neq p-2$ et $\mathcal{C}_{\mathbf{Q}(\zeta_p)}(\omega^{-i}) = 0$. Par exemple $i=1$ convient par le théorème d'Herbrand--Ribet puisque $B_2 = \frac{1}{6}$ est premier à $p\geq 5$. 

On a (sous les hypothèses sur $i$) $\alpha_0(\omega^{-i}) = j(\omega^{-i})=1$.
En effet, il s'agit de voir par le théorème~\ref{genus} que $n(\omega^{-i})=0$. Il suffit de voir que $\mathbf{Z}[\zeta_p]^{\times}(\omega^{-i})=0$. Cela découle du résultat classique suivant :

\begin{lem}\label{cyclotomic_unit}
Soit $n>1$ un entier impair. Soit $U_n$ (resp. $U_n^+$) le groupe des unités de l'anneau des entiers de $\mathbf{Q}(\zeta_n)$ (resp. du sous-corps totalement réel maximal $\mathbf{Q}(\zeta_n)^+$ de $\mathbf{Q}(\zeta_n)$). 

Si $n$ est un nombre premier impair, on a :
$$U_n = \zeta_n^{\mathbf{Z}} \cdot U_n^+ \text{ .}$$
En général, sans supposer que $n$ est premier, 
$$U_n/(U_n^+ \cdot \zeta_n^{\mathbf{Z}})$$
est un $2$-groupe d'exposant divisant $2$.
\end{lem}
\begin{proof}
Le cas $n$ premier impair est bien connu, mais nous n'avons pas trouvé de référence pour le cas $n$ impair quelconque. Nous reproduisons une preuve d'Emerton trouvée sur le site de discussion en ligne StackExchange.

Soit $c \in \text{Gal}(\mathbf{Q}(\zeta_n)/\mathbf{Q}(\zeta_n)^+)$ non trivial ($c$ est la conjugaison complexe).
Soient $\epsilon \in U_n$ et $u= \frac{\epsilon}{c(\epsilon)}$. Soit $U_n^{-} = \{x \in U_n, c(x) = x^{-1}\}$, on a alors $u \in U_n^{-}$. Le théorème des unités de Dirichlet montre que $U_n^+$ et $U_n$ ont le même rang. Comme $U_n^{-} \cap U_{n}^+ = \{1,-1\}$, le groupe $U_n^{-}$ est fini, donc tout élément de $U_n^-$ est une racine de l'unité dans $\mathbf{Q}(\zeta_n)$. Le groupe des racines de l'unités $\mu(\mathbf{Q}(\zeta_n))$ de $\mathbf{Q}(\zeta_n)$ est égal à $(-1)^{\mathbf{Z}}\cdot (\zeta_n)^{\mathbf{Z}}$. En effet, soit $m$ le cardinal de $\mu(\mathbf{Q}(\zeta_n))$. On a alors $2n \mid m$ et $\varphi(m)$ divise $\varphi(n)$ (car $\mathbf{Q}(\zeta_m)$ est un sous-corps de $\mathbf{Q}(\zeta_n)$). Cela implique que $m=2n$.
Par conséquent, $u = \pm \zeta_n^{\alpha}$ pour un $\alpha \in \mathbf{Z}$. Donc $u = \pm v^2$ pour un $v \in (\zeta_n)^{\mathbf{Z}}$. Soit $\gamma = \epsilon \cdot v^{-1}$. On a $\gamma = \pm c(\epsilon) \cdot v = \pm c(\gamma)$, donc $\gamma^2 \in U_n^+$, donc $\epsilon^2 \in U_n^{+} \cdot (\zeta_n)^{\mathbf{Z}}$, donc $U_n/(U_n^+ \cdot \zeta_n^{\mathbf{Z}})$ est d'exposant divisant $2$.

Si $n$ est premier, on a en fait $c(\gamma)=\gamma$ car sinon $c(\gamma) = - \gamma$ et dans ce cas $\gamma \not\in \mathbf{Q}(\zeta_n)^+$ mais $\gamma^2 \in \mathbf{Q}(\zeta_n)^+$, donc $\mathbf{Q}(\zeta_n) = \mathbf{Q}(\zeta_n)^+(\sqrt{u})$ est non ramifié en dehors de $2$, ce qui est impossible car si $n$ est premier, $\mathbf{Q}(\zeta_n)$ est totalement ramifié en $n$ sur $\mathbf{Q}$, donc sur $\mathbf{Q}(\zeta_n)^+$.
\end{proof}

Nous allons calculer le noyau $U_0(\omega^{-i})$ de la flèche $I_0(\omega^{-i}) \rightarrow I_1(\omega^{-i+1})$. Comme $\alpha_0(\omega^{-i})=1$, alors $\alpha_i(\chi_0) \leq 1$ est non nul seulement si $U_0(\omega^{-i})=0$. 

Soit $\varphi : \mathcal{C}_{K_1 }\rightarrow  \mathcal{C}_{K_1} / ((I_{S} \cdot\mathcal{C}_{K_1})\cdot  \mathcal{C}_{K_1}^p) = I_0$. La proposition suivante est due à Iimura si $i=1$. Nous reprenons sa démonstration dans le cas général. Notons que dans l'énoncé on ne suppose pas les conditions précédentes sur $i$.

\begin{prop}(\cite[Lemma $2.1$]{Iimura} pour le cas $i=1$)\label{kernel}
Soit $1 \leq i \leq p-1$.
Soit $\mathcal{C}_1 \subset \mathcal{C}_{K_1}$ le sous-groupe engendré par les classes d'idéaux $\mathfrak{a}$ tels que $\sigma(\mathfrak{a})=\mathfrak{a}$ (rappelons que $\sigma$ est un générateur de $S = Gal(K_1/\mathbf{Q}(\zeta_p))$). Alors 
$$\varphi(\mathcal{C}_1(\omega^{-i})) \subset U_0(\omega^{-i})$$
avec égalité si $\mathcal{C}_{\mathbf{Q}(\zeta_p)}(\omega^{-i+1}) = 0$ (ce qui est vrai si la conjecture de Vandiver est vraie ou si $i=1$) et $ (\mathbf{Z}[\zeta_p]^{\times} \cap \text{N}_{K_1/\mathbf{Q}(\zeta_p)}(K^{\times}))(\omega^{-i+1}) =   (\text{N}_{K_1/\mathbf{Q}(\zeta_p)}(\mathcal{O}_{K_1}^{\times}))(\omega^{-i+1})$ où $\mathcal{O}_{K_1}$ est le groupe des unités de $K_1$ (ce qui est vrai si $i=1$).
\end{prop}

\begin{proof}
L'inclusion $\varphi(\mathcal{C}_1(\omega^{-i})) \subset U_0(\omega^{-i})$ est immédiate car $\Delta \in I_S$. Montrons l'assertion sur le cas d'égalité. On suppose donc dans toute la suite de la démonstration que $$\mathcal{C}_{\mathbf{Q}(\zeta_p)}(\omega^{-i+1}) = 0$$ et que $$(\mathbf{Z}[\zeta_p]^{\times} \cap \text{N}_{K_1/\mathbf{Q}(\zeta_p)}(K^{\times})/ \text{N}_{K_1/\mathbf{Q}(\zeta_p)}(\mathcal{O}_{K_1}^{\times}))(\omega^{-i+1})=0 \text{ .}$$
Notons $$\mathcal{C}_1' = \{c \in \mathcal{C}_{K_1}, \Delta \cdot c=1\} =  \{c \in \mathcal{C}_{K_1}, \sigma(c)=c\}$$ et $$\mathcal{C}_1'' =  \{c \in \mathcal{C}_{K_1}, \Delta \cdot c \in  \delta \cdot \mathcal{C}_{K_1}\} \text{ .}$$ 
\begin{lem}
On a $\mathcal{C}_1'(\omega^{-i}) = \mathcal{C}_1''(\omega^{-i})$.
\end{lem}
\begin{proof}
Il s'agit de montrer que $\mathcal{C}_1''(\omega^{-i}) \subset \mathcal{C}_1'(\omega^{-i})$. Soit $c \in \mathcal{C}_{K_1}$ tel que $\Delta\cdot c = \delta\cdot c'$ avec $c' \in \mathcal{C}_{K_1}$. Alors $\Delta \cdot (e_{\omega^{-i}} \cdot c) = e_{\omega^{-i} \cdot \omega} \cdot( \Delta \cdot c )= e_{\omega^{-i} \cdot \omega}\cdot (\delta \cdot c') \in \mathcal{C}_{\mathbf{Q}(\zeta_p)}(\omega^{-i+1})$. Si $\mathcal{C}_{\mathbf{Q}(\zeta_p)}(\omega^{-i+1}) = 0$ alors on a $e_{\omega^{-i}} \cdot c \in \mathcal{C}_1'(\omega^{-i})$ et donc $\mathcal{C}_1''(\omega^{-i}) \subset \mathcal{C}_1'(\omega^{-i})$. \end{proof}

\begin{lem}
On a $\mathcal{C}_1(\omega^{-i}) = \mathcal{C}_1'(\omega^{-i})$.
\end{lem}
\begin{proof}
On sait, par \cite[Proposition $3$ (i)]{Jaulent} que $(\mathcal{C}_1'/\mathcal{C}_1)(\omega^{-i}) \simeq (\mathbf{Z}[\zeta_p]^{\times} \cap \text{N}_{K_1/\mathbf{Q}(\zeta_p)}(K^{\times})/ \text{N}_{K_1/\mathbf{Q}(\zeta_p)}(\mathcal{O}_{K_1}^{\times}))(\omega^{-i+1})$, ce dernier groupe étant trivial par hypothèse.
\end{proof}
Finissons la preuve de la proposition~\ref{kernel}. \\
Soit $\hat{\mathcal{C}}_1 = \{c \in \mathcal{C}_{K_1}, \Delta\cdot c \in \mathcal{C}_{K_1}^p\}$. On a $U_0=\varphi(\hat{\mathcal{C}}_1)$, donc $U_0(\omega^{-i}) = \varphi(\hat{\mathcal{C}}_1(\omega^{-i}))$.
\begin{lem}
On a $\varphi(\hat{\mathcal{C}}_1) = \varphi(\mathcal{C}_1'')$.
\end{lem}
\begin{proof}
Si $c \in \hat{\mathcal{C}}_1$, alors $\Delta \cdot c = c'^p$ pour un $c' \in \mathcal{C}_{K_1}$. Par le lemme~\ref{delta}, cela implique qu'il existe $c_1$ et $c_2$ dans $\mathcal{C}_{K_1}$ tels que $\Delta \cdot (c\cdot (\Delta^{p-2}\cdot c_1)) = c_2^{\delta}$, donc $\varphi(c) = \varphi(c\cdot (\Delta^{p-2}\cdot c_1)) \in \varphi(\mathcal{C}_1'')$. Réciproquement, si $c \in \mathcal{C}_1''$, alors $\Delta\cdot c = \delta\cdot c'$ pour un certain $c' \in \mathcal{C}_{K_1} $. Par le lemme~\ref{delta}, il existe $c_1$ et $c_2$ dans $\mathcal{C}_{K_1}$ tels que $\Delta\cdot (c \cdot (\Delta^{p-2}\cdot c_1)) = c_2^p$, donc $\varphi(c) = \varphi(c \cdot (\Delta^{p-2}\cdot c_1)) \in \varphi(\hat{\mathcal{C}}_1)$.
\end{proof}

La proposition~\ref{kernel} résulte du lemme précédent.
\end{proof}

\section{Un critère par les logarithmes d'unité (d'une somme de Gauss)}

Nous allons maintenant expliciter le terme $\varphi(\mathcal{C}_1(\omega^{-i}))$. Soit $\mathfrak{p}$ un idéal premier de $\mathbf{Q}(\zeta_p)$ au-dessus de $(N)$ (comme $N \equiv 1  \text{ (mod } p\text{)}$, $N$ est totalement décomposé dans $\mathbf{Q}(\zeta_p)$). Soit $\mathfrak{P}$ l'unique idéal premier au-dessus de $\mathfrak{p}$ dans $K_1$. Si $\mathfrak{a}$ est un idéal de $K_1$, soit $c(\mathfrak{a})$ sa classe dans le groupe des classes $C(K_1)$ de $K_1$.
Considérons la classe
$$\gamma := \prod_{t \in T} c(t(\mathfrak{P}))^{\omega^i(t)} \in \mathcal{C}_{K_1}(\omega^{-i}) $$
où $t(\mathfrak{P})$ est l'image de $\mathfrak{P}$ par l'élément $t$ de $T = Gal(K_1/K)$. Il est clair que $\mathcal{C}_1(\omega^{-i})$ est engendré par $\gamma$. 
On a :
$$\text{N}_{K_1/\mathbf{Q}(\zeta_p)}(\gamma) = \prod_{t \in T} c(t(\mathfrak{p}))^{\omega^i(t)} \in \mathcal{C}_{\mathbf{Q}(\zeta_p)}(\omega^{-i}) \text{.}$$
Comme par hypothèse sur $i$, on a  $\mathcal{C}_{\mathbf{Q}(\zeta_p)}(\omega^{-i}) = 0$, alors $\text{N}_{K_1/\mathbf{Q}(\zeta_p)}(\gamma)$ est trivial.

Comme $\mathbf{Z}_p$ est plat sur $\mathbf{Z}$, on a une suite exacte :
$$1 \rightarrow \mathbf{Z}[\zeta_p]^{\times} \otimes_{\mathbf{Z}} \mathbf{Z}_p \rightarrow \mathbf{Q}(\zeta_p)^{\times} \otimes_{\mathbf{Z}} \mathbf{Z}_p \rightarrow \mathcal{I} \otimes_{\mathbf{Z}} \mathbf{Z}_p \rightarrow C(\mathbf{Q}(\zeta_p)) \otimes_{\mathbf{Z}} \mathbf{Z}_p \rightarrow 1$$
où $C(\mathbf{Q}(\zeta_p))$ est le groupe des classes de $\mathbf{Q}(\zeta_p)$ et $\mathcal{I}$ est le groupe des idéaux fractionnaires de $\mathbf{Q}(\zeta_p)$.
L'élément $\prod_{t \in T} t(\mathfrak{p})^{\omega^i(t)}$ de $\mathcal{I} \otimes_{\mathbf{Z}} \mathbf{Z}_p$ a son image triviale dans $C(\mathbf{Q}(\zeta_p)) \otimes_{\mathbf{Z}} \mathbf{Z}_p$, donc il existe $u_{\omega^{-i}} \in (\mathbf{Q}(\zeta_p)^{\times} \otimes_{\mathbf{Z}} \mathbf{Z}_p)(\omega^{-i})$ tel que :
$$ \prod_{t \in T} t(\mathfrak{p})^{\omega^i(t)} = (u_{\omega^{-i}}) \text{ .}$$

Comme $\mathbf{Z}[\zeta_p]^{\times}(\omega^{-i}) = 0$ par le lemme~\ref{cyclotomic_unit} et le fait que $\omega^{-i}$ est impair, $u_{\omega^{-i}}$ est uniquement déterminé.

En fait $u_{\omega^{-i}}$ est lié à une somme de Gauss. 

Rappelons qu'on a fixé une racine primitive $N$-ième $\zeta_N \in \overline{\mathbf{Q}}$. Soit $\chi_N : (\mathbf{Z}/N\mathbf{Z})^{\times} \rightarrow \mathbf{Z}_N^{\times}$ le caractère de Teichmüller (qui vérifie $\chi_N(a) \equiv a \text{ (mod }N\text{)}$). On peut voir $\chi_N^{\frac{N-1}{p}}$ comme un caractère à valeurs dans $\mathbf{Z}[\zeta_p]$ par le choix de $\mathfrak{p}$.

Considérons la somme de Gauss : 
$$\mathcal{G} = -\sum_{a=1}^{N-1} \chi_N(a)^{\frac{-(N-1)}{p}} \cdot \zeta_N^a \in \mathbf{Q}(\zeta_p, \zeta_N)  \text{ .}$$
Rappelons qu'on a noté $K_0 = \mathbf{Q}(\zeta_p, \zeta_N^+)$. Par la théorie de Galois, $\mathcal{G}^p \in \mathbf{Q}(\zeta_p)$, donc $\mathcal{G} \in K_0$. Soit $\mathfrak{P}_0$ l'idéal premier au-dessus de $\mathfrak{p}$ dans $K_0$.

Le groupe $T = \text{Gal}(\mathbf{Q}(\zeta_p)/\mathbf{Q})$ s'identifie à $\text{Gal}(K_0/\mathbf{Q}(\zeta_N^+))$. Si $t \in T$, notons $\text{Rep}(t)$ l'unique entier $a \in \{1,2,...,p-1\}$ tel que $t = \sigma_a$ (\textit{i.e.} $t(\zeta_p) = \zeta_p^a$).

On a la factorisation en idéaux premiers dans $K_0$ : 
$$(\mathcal{G}) = \prod_{t \in T} t^{-1}(\mathfrak{P}_0)^{\text{Rep}(t)} \text{ .} $$
Soit $$\mathcal{G}_i = \prod_{t \in T} t(\mathcal{G})^{\omega^i(t)} \in (K_0^{\times} \otimes_{\mathbf{Z}} \mathbf{Z}_p)(\omega^{-i}) \text{ .}$$
Si $\mathfrak{Q}$ est un idéal premier de $K_0$, et $x \otimes a \in K_0^{\times} \otimes_{\mathbf{Z}} \mathbf{Z}_p$, on pose $v_{\mathfrak{Q}}(x \otimes a) = a\cdot v_{\mathfrak{Q}}(x) \in \mathbf{Z}_p$ (où $v_{\mathfrak{Q}}(x)$ est la valuation $\mathfrak{Q}$-adique de $x$). Cela induit une injection $K_0^{\times} \otimes_{\mathbf{Z}} \mathbf{Z}_p \rightarrow \prod_{\mathfrak{Q}} \mathbf{Z}_p$ où $\mathfrak{Q}$ décrit les idéaux premiers de $K_0$. Par platitude de $\mathbf{Z}_p$ sur $\mathbf{Z}$, le noyau est $\mathcal{O}_{K_0}^{\times} \otimes_{\mathbf{Z}} \mathbf{Z}_p$ où $\mathcal{O}_{K_0}$ est l'anneau des entiers de $K_0$. On a  $\mathcal{G}_i \in e_{\omega^{-i}}(K_0^{\times} \otimes_{\mathbf{Z}} \mathbf{Z}_p)$ (la $\omega^{-i}$-composante de $K_0^{\times} \otimes_{\mathbf{Z}} \mathbf{Z}_p$).

\begin{lem}
On a $e_{\omega^{-i}}(\mathcal{O}_{K_0}^{\times} \otimes_{\mathbf{Z}} \mathbf{Z}_p) = 0$. 
\end{lem}
\begin{proof}
On a $K_0 \hookrightarrow \mathbf{Q}(\zeta_{Np})$. Par platitude de $\mathbf{Z}_p$ sur $\mathbf{Z}$, il suffit de montrer le même résultat avec $\mathbf{Q}(\zeta_{Np})$ à la place de $K_0$. Comme $\omega^{-i}$ est impair et que $\omega^{-i} \neq \omega$, cela découle du lemme~\ref{cyclotomic_unit}.
\end{proof}

Le résultat suivant est important car il nous permet de lier $u_{\omega^{-i}}$ à une somme de Gauss, qu'on calculera avec la formule de Gross--Koblitz.
\begin{prop}\label{unit_gauss}
On a $$u_{\omega^{-i}}^{B_{1, \omega^i}} = \mathcal{G}_i \text{ .}$$
\end{prop}
\begin{proof}
Soit $v = u_{\omega^{-i}}^{B_{1, \omega^i}}$.
Par le lemme précédent, il suffit de vérifier que pour tout $\mathfrak{Q}$ idéal premier de $K_0$ on a $v_{\mathfrak{Q}}(\mathcal{G}_i) = v_{\mathfrak{Q}}(v)$. 
On a $v_{\mathfrak{Q}}(\mathcal{G}_i) = v_{\mathfrak{Q}}(v) = 0$ sauf si $\mathfrak{Q} = t(\mathfrak{P}_0)$ pour un $t \in T$, auquel cas on a $v_{t(\mathfrak{P}_0)}(\mathcal{G}_i) =  v_{t(\mathfrak{P}_0)}(v)  =  \sum_{\tau \in T} \omega^i(\tau)\cdot Rep(\tau\cdot t^{-1})$.
\end{proof}

Définissons un logarithme discret de $\mathbf{Q}_N^{\times} \otimes_{\mathbf{Z}} \mathbf{Z}_p$ à valeur dans $\mathbf{Z}/p^{\nu}\mathbf{Z}$.
Soit $\Lambda : \mathbf{Q}_N^{\times} \otimes_{\mathbf{Z}} \mathbf{Z}_p \rightarrow \mathbf{Z}/p^{\nu}\mathbf{Z}$ défini par :
$$\Lambda(N^k \cdot u  \otimes_{\mathbf{Z}} a) = \overline{a}\cdot  \log(\overline{u})  \text{ (mod } p^{\nu}\text{)}$$
où $k \in \mathbf{Z}$, $u \in \mathbf{Z}_N^{\times}$, $\overline{u}$ et $\overline{a}$ désignent la réduction modulo $N$ et $p^{\nu}$ respectivement.

Comme $N$ est totalement décomposé dans $\mathbf{Q}(\zeta_p)$, la complétion $\mathbf{Q}(\zeta_p)_{\mathfrak{p}}$ de $\mathbf{Q}(\zeta_p)$ en $\mathfrak{p}$ est égale à $\mathbf{Q}_N^{\times}$.
Rappelons que $\varphi(\mathcal{C}_1(\omega^{-i})) \subset U_0(\omega^{-i})$ et que $U_0(\omega^{-i})$ est un sous-$\mathbf{F}_p$-espace vectoriel de $I_0(\omega^{-i})$ de dimension $\leq 1$ (car $I_0(\omega^{-i})$ est de dimension $\leq 1$).

La proposition suivante est le point essentiel de cette partie.

\begin{prop}\label{critere_zero}
On a 
$$\varphi(\mathcal{C}_1(\omega^{-i})) = 0$$
si et seulement si on a
$$\Lambda(\frac{u_{\omega^{-i}}}{N}) \equiv 0 \text{ (mod }p\text{)} \text{ .}$$
\end{prop}
\begin{proof}
Comme $N_{K_1/\mathbf{Q}(\zeta_p)}(\gamma) = 1$, on $\gamma \in \overline{\mathcal{C}}_{K_1}(\omega^{-i})$ et  $\varphi(\mathcal{C}_1(\omega^{-i})) = \varphi(\gamma)^{\mathbf{Z}_p}$ est trivial si et seulement si $\psi(\gamma) \in J(\omega^{-i})$ est nul.

Pour $t \in T$, soit 
$$\chi_t : \mathbf{Q}(\zeta_p)_{t(\mathfrak{p})}  \rightarrow Gal((K_1)_{t(\mathfrak{P})}/\mathbf{Q}(\zeta_p)_{t(\mathfrak{p})}) = S$$
l'application d'Artin locale de la théorie locale du corps de classe. 
Soit
$$D : \mathbf{Q}(\zeta_p)^{\times} \rightarrow S^{(T)}$$
donnée par $x \mapsto (\chi_t(x))_{t \in T}$ et $\hat{S} := D(\mathbf{Z}[\zeta_p]^{\times}) \subset S^{(T)}$.
Notons que le noyau la restriction de $D$ à $\mathbf{Z}[\zeta_p]^{\times}$ est $\mathbf{Z}[\zeta_p]^{\times} \cap \text{N}_{K_1/\mathbf{Q}(\zeta_p)}(K_1^{\times})$. En effet, par le théorème des normes de Hasse pour une extension cyclique, un élément $x$ de $\mathbf{Z}[\zeta_p]^{\times}$ est une norme de $K_1$ si et seulement si c'est une norme localement en tous les idéaux premiers, si et seulement si c'est une norme en tous les idéaux premiers ramifiés (car $[K_1 : \mathbf{Q}(\zeta_p)] = p$ est premier), si et seulement si $D(x)=0$ par la théorie locale du corps de classe. En particulier, 
$$\hat{S} \simeq \mathbf{Z}[\zeta_p]^{\times}/\mathbf{Z}[\zeta_p]^{\times} \cap \text{N}_{K_1/\mathbf{Q}(\zeta_p)}(K_1^{\times}) \text{ .}$$

Soit $$\hat{\chi} : J \rightarrow S^{(T)}/\hat{S}$$
l'application suivante. Si $c \in J$, $N_{K_1/\mathbf{Q}(\zeta_p)}(c)=1$. Si on écrit $c = c(\mathfrak{a})$ comme la classe de la norme d'un idéal fractionnaire $\mathfrak{a}$ de $K_1$. Alors $N_{K_1/\mathbf{Q}(\zeta_p)}(\mathfrak{a})$ est principal, engendré par un élément $u \in \mathbf{Q}(\zeta_p)^{\times}$. On pose 
$$\hat{\chi}(\mathfrak{a}) = (\chi_t(u))_{t \in T} \text{ .}$$
Cela ne dépend pas du choix de $\mathfrak{a}$ car on a quotienté par $\hat{S}$.

Le groupe $S^{(T)}$ est un $T$-module naturel (si $\tau \in T$, $\tau\cdot (x_t)_{t \in T} := (x_{\tau \cdot t})_{t \in T}$).
L'application $D$ est $T$-équivariante. Donc $\hat{S}$ est un sous $T$-module de $S^{(T)}$, et $\hat{\chi}$ est $T$-équivariante.

\begin{lem}\label{Halter_Koch}\cite[Satz 1]{Halter_Koch}
L'application $\hat{\chi}$ est injective. C'est un isomorphisme si $\nu = 1$ (i.e. si $p \mid \mid N-1$ ou encore si $(1-\zeta_p)$ est ramifié dans $K_1$). 
\end{lem}

Le lemme~\ref{cyclotomic_unit} montre que $\hat{S}(\omega^{-i})=0$ (sous les hypothèses sur $i$). On a donc $\varphi(\gamma) = 0$ si et seulement si $\hat{\chi}(\omega^{-i})(u_{\omega^{-i}}) = 0$.

\begin{lem}\label{Hilbert_symbol}
On a $\hat{\chi}(\omega^{-i})(u_{\omega^{-i}}) = 0$ si et seulement si $\Lambda(\frac{u_{\omega^{-i}}}{N}) = 0 \text{ (mod }p\text{)}$.
\end{lem}
\begin{proof}
On a $\hat{\chi}(\omega^{-i})(u_{\omega^{-i}}) = 0$ si et seulement si $(u_{\omega^{-i}}, N)_{\mathfrak{p}}=1$ où si $a,b \in \mathbf{Q}(\zeta_p)_{\mathfrak{p}}^{\times}$, $(a,b)_{\mathfrak{p}} \in \mu_p$ désigne le symbole de Hilbert relativement au corps $\mathbf{Q}(\zeta_p)_{\mathfrak{p}}$. On a la formule classique : 
$$(a,b) = \Legendre((-1)^{v(a)v(b)} \cdot \frac{b^{v(a)}}{a^{v(b)}}, \mathfrak{p}) $$
où $v(a)$ et $v(b)$ désignent la valuation $\mathfrak{p}$-adique de $a$ et $b$ respectivement, et si $v(x)=0$, 
$$\Legendre(x, \mathfrak{m}) = x^{\frac{N-1}{p}} \text{(mod }\mathfrak{p} \text{) } \in \mu_p \subset \mathcal{\kappa}^{\times} $$
($\kappa = \mathbf{F}_N$ est le corps résiduel de $\mathbf{Q}(\zeta_p)$ en $\mathfrak{p}$). 
Cela achève la preuve du lemme~\ref{Hilbert_symbol} et de la proposition~\ref{critere_zero}.
\end{proof}
\end{proof}

La preuve de la proposition~\ref{critere_zero} donne le : 
\begin{lem}\label{j(w)}
On a $j(\omega)=0$.
\end{lem}
\begin{proof}
On sait que c'est vrai si $\nu>1$ par la proposition~\ref{genus}.

Si $\nu=1$, il s'agit de voir que $\zeta_p \not\in \mathbf{Z}[\zeta_p]^{\times} \cap \text{N}_{K_1/\mathbf{Q}(\zeta_p)}(K_1^{\times})$. On a vu dans la preuve de la proposition~\ref{critere_zero} que $\zeta_p \in  \mathbf{Z}[\zeta_p]^{\times} \cap \text{N}_{K_1/\mathbf{Q}(\zeta_p)}(K_1^{\times})$ si et seulement si $D(\zeta_p)=0$ si et seulement si $\Lambda(\zeta_p)=0$, si et seulement si $\zeta_p$ est une puissance $p$-ième modulo $N$, ce qui est le cas si et seulement si $\nu>1$.
\end{proof}

Pour montrer le théorème~\ref{class_group}, il suffit de montrer le lemme~\ref{j(w)} ci-dessous et l'identité :
\begin{thm}\label{identite_u}
On a : 
$$\Lambda(\frac{{u}_{\omega^{-i}}}{N}) \equiv \frac{1}{B_{1,\omega^i}} \cdot \left( \sum_{k=1}^{N-1} \left( \sum_{a=1}^{k-1} \omega^i(a)\right)\cdot  \log(k)\right)\text{ (mod }p^{\nu}\text{)} \text{ .}$$
\end{thm}
Nous montrons cette identité dans la dernière partie en utilisant la formule de Gross--Koblitz, en même temps que les diverses identités annoncées dans l'introduction.
\\ \\
Traitons le cas $p=5$. Dans ce cas, $\mathbf{Q}(\zeta_5)$ est principal, son groupe des unités est $E= (-\zeta_5)^{\mathbf{Z}} \cdot (1+\zeta_5)^{\mathbf{Z}}$. Notons que $E\otimes_{\mathbf{Z}} \mathbf{Z}_5 = \mathbf{Z}_5(2)$ (c'est à dire que l'action de $T = \text{Gal}(\mathbf{Q}(\zeta_5)/\mathbf{Q})$ est donnée par $\omega^2$). En effet, $1+\zeta_5^2 = -\zeta_5^{-1}\cdot (1+\zeta_5)^{-1}$ et $-1 \equiv 2^2 \text{ (mod } 5\text{)}$.

On a, par le lemme~\ref{decomposition}, $r_K = \sum_{i=0}^{3} \alpha_i(\chi_0)$ avec $\alpha_i(\chi_0) \in \{0,1\}$ et $\alpha_i(\chi_0) \leq \alpha_0(\omega^{-i})$. On a vu que $\alpha_0(\chi_0)=1$, et que $\alpha_1(\chi_0)=1$ si et seulement si $\sum_{k=1}^{\frac{N-1}{2}} k \cdot \text{log}(k) \equiv 0 \text{ (mod }p \text{)}$.

On a $\alpha_3(\chi_0) = \alpha_0(\omega) = 0$.
En effet, il suffit de montrer que $\alpha_0(\omega) = 0$ car $\alpha_3(\chi_0) \leq  \alpha_0(\omega^{-3}) =  \alpha_0(\omega)$. Cela découle du lemme~\ref{j(w)}.

Il reste à étudier $\alpha_2(\chi_0) \in \{0,1\}$. Comme $\alpha_2(\chi_0)\leq \alpha_0(\omega^2)$, une condition nécessaire pour que $\alpha_2(\chi_0) = 1$ est que $\alpha_0(\omega^2)=1$, \textit{i.e.} que $\mathbf{Z}[\zeta_5]^{\times}(\omega^2) = (\mathbf{Z}[\zeta_5]^{\times} \cap \text{N}_{K_1/\mathbf{Q}(\zeta_p)}(K_1^{\times}))(\omega^2)$ par la proposition~\ref{genus}. Cette dernière condition est équivalente à $$\sum_{a=1}^{4} \omega^2(a)\cdot \text{log}(\sigma_a(1+\zeta_5)) \equiv 0 \text{ (mod }p \text{)} $$
par la preuve de la proposition~\ref{critere_zero}.

Supposons cette condition réalisée. On a alors $I_0(\omega^2) \simeq \mathbf{F}_5$. La proposition~\ref{kernel} pour $i=2$ et la proposition~\ref{critere_zero} (dont la preuve s'adapte pour $i=2$ car $\hat{S}(\omega^{-i})=0$ par la condition précédente) montrent que le noyau de la flèche $I_0(\omega^2) \rightarrow I_1(\omega^3)$ est nul si et seulement si 
$$\sum_{a=2}^{4} (a^2-1)\cdot \emph{log}(\sigma_a(u)) \equiv 0 \text{ (mod } 5 \text{)}$$
où $(u)$ est un idéal premier au-dessus de $N$ dans $\mathbf{Q}(\zeta_5)$ (cette dernière quantité ne dépend pas du choix de $u$ car $\sum_{a=1}^{4} (a^2-1)\cdot \text{log}(\sigma_a(1+\zeta_5)) \equiv 0 \text{ (mod }p \text{)}$). 

Cela achève la preuve du théorème~\ref{p=5}. Nous voyons que la difficulté pour déterminer complètement $r_K$  dans le cas $p=5$ est de calculer le noyau de la flèche $I_1(\omega^3) \rightarrow I_2(\chi_0)$. Nous conjecturons que ce noyau est nul si et seulement si $\sum_{k=1}^{\frac{N-1}{2}} k\cdot \text{log}(k) \equiv0 \text{ (mod } 5 \text{)}$ (quand les conditions $I_0(\omega) = I_1(\omega^2) \simeq \mathbf{F}_p$ sont satisfaites). Cela impliquerait la réciproque du théorème~\ref{CE} pour $p=5$, \textit{i.e.} que $r_K>1$ si et seulement si $\sum_{k=1}^{\frac{N-1}{2}} k\cdot \text{log}(k) \equiv0\text{ (mod } 5 \text{)}$.

\section{Etude de $r_{K_0}$}

\begin{thm}\label{critere_r_0}
On a $\sum_{k=1}^{\frac{N-1}{2}} k\cdot  \log(k) \equiv 0 \text{ (mod }p\text{)}$ si et seulement si $r_{K_0}(\omega^{-1})>1$.
\end{thm}
\begin{proof}
On refait la même étude avec $K_0$ au lieu de $K_1$. Cette fois $Gal(K_0/\mathbf{Q})$ est abélien donc $S=Gal(K_0/\mathbf{Q}(\zeta_p)) \simeq \mathbf{F}_p $ et $T=Gal(K_0/\mathbf{Q}(\zeta_N^+)) = \mathbf{F}_p^{\times}$ commutent. On a donc une suite exacte :
$$ 1\rightarrow U_0(\omega^{-1}) \rightarrow I_0(\omega^{-1}) \rightarrow I_1(\omega^{-1}) \rightarrow 1$$
où les définitions de $U$ et de $I_i$ sont similaires pour $K_0$. L'analogue immédiat de la proposition~\ref{sansnom1} (avec $\chi=\omega^{-1}$) donne (en tenant compte du fait que $\mathcal{C}_{\mathbf{Q}(\zeta_p)}(\omega^{-1})=0$ par le théorème d'Herbrand--Ribet) :
$$I_0(\omega^{-1}) \simeq \mathbf{F}_p \text{ .}$$
(en fait $I_0$ correspond par la théorie globale du corps de classe à l'extension abélienne $K_{-1}\cdot K_0$ de $K_0$ où $K_{-1}$ est l'unique $\mathbf{F}_p$-extension de $\mathbf{Q}(\zeta_p)$ non ramifiée en dehors de $N$, sur laquelle le groupe de Galois $T$ agit par $\omega^{-1}$, \textit{cf.} (iii) de  \cite[Proposition $5.4$]{CE}). Donc $U_0(\omega^{-1}) = 0$ si et seulement si $I_1(\omega^{-1}) \simeq \mathbf{F}_p$ si et seulement si $r_{K_0}(\omega^{-1})>1$ (car si $I_1(\omega^{-1}) = 0$, alors pour tout $i\geq 1$, $I_i(\omega^{-1})=0$). L'analogue de la proposition~\ref{kernel} pour $i=1$ donne : 
$$U_0(\omega^{-1}) = \varphi(\mathcal{C}_0(\omega^{-1})) $$
(on utilise le fait que $\mathcal{C}_{\mathbf{Q}(\zeta_p)}(\omega^{-1})=0$ et que $\mathbf{Z}[\zeta_p]^{\times}[\omega^{0}]=0$).
Ici $\mathcal{C}_0$ est engendré par les $p-1$ idéaux premiers de $K_0$ au dessus de $N$, qu'on note comme ci-dessus $t(\mathfrak{P})$, $t \in T = \text{Gal}(\mathbf{Q}(\zeta_p)/\mathbf{Q}) = \text{Gal}(K_0/\mathbf{Q}(\zeta_N^+))$. L'analogue immédiat de la proposition~\ref{critere_zero} est le suivant :

\begin{prop}\label{critere_zero_r_0}
Soit $\beta \in \mathbf{Q}(\zeta_p)$ tel que $K_0 = \mathbf{Q}(\zeta_p)(\beta^{\frac{1}{p}})$ (c'est possible par la théorie de Kummer). Alors on a 
$$\varphi(\mathcal{C}_0(\omega^{-1})) = 0$$
si et seulement si
$$\Lambda(\frac{u_{\omega^{-1}}}{\beta}) \equiv 0 \text{ (mod }p\text{)} \text{ .}$$
\end{prop}

On va montrer dans la partie suivante le :
\begin{thm}\label{identite_garett}
On a $$\Lambda(\frac{u_{\omega^{-1}}}{\beta}) \equiv \Lambda(\frac{u_{\omega^{-1}}}{N}) \equiv \frac{1}{B_{1,\omega}}\cdot \mathcal{S}_1 \equiv -8 \cdot \sum_{k=1}^{\frac{N-1}{2}} k\cdot  \log(k) \text{ (mod }p\text{)} \text{ .}$$
\end{thm}
Cela achève la preuve du théorème~\ref{critere_r_0}.
\end{proof}

Esquissons finalement la preuve du théorème~\ref{r_0>i}, implicite dans \cite{Schoof}. On raisonne par récurrence sur $1 \leq i \leq p-1 $. Si $i=1$, cela découle du théorème~\ref{r_0}. Nous nous référons à la partie suivante (ou à \cite[p. $190$]{Schoof}) pour la définition de l'élément de Stickelberger tordu $\phi_{\omega^{-1}}$. Si $a \in (\mathbf{Z}/N\mathbf{Z})^{\times}$, le coefficient de $\phi_{\omega^{-1}}$ en $[a^{-1}]$ est $B_{1, \omega} + \sum_{k=1}^{a-1} \omega(k)$  par la preuve de la proposition~\ref{Stickelberger} dans la partie ci-dessous. Si on identifie un générateur de $(\mathbf{Z}/N\mathbf{Z})^{\times}$ à l'indéterminée $1+T$ (modulo l'idéal engendré par $(1+T)^{N-1}-1$, on peut voir $\phi_{\omega^{-1}}$ comme un élément de $\mathbf{Z}_p[T]/((1+T)^{N-1}-1))$. Son coefficient modulo $p$ en $T^j$ pour $1 \leq j < i$ est nul par l'hypothèse de récurrence. Son coefficient en $T^i$ est $\frac{-4}{3} \cdot \sum_{k=1}^{\frac{N-1}{2}} k\cdot \text{log}(k)^i$ par la même preuve que le lemme~\ref{B_2} (avec l'hypothèse que $\sum_{k=1}^{\frac{N-1}{2}} k\cdot \text{log}(k)^j \equiv 0 \text{ (mod }p\text{)}$ pour tout $1 \leq j < i$). Comme $\phi_{\omega^{-1}}$ est dans l'idéal de Fitting du $\mathbf{Z}_p[[T]]$-module $\mathcal{C}_{K_0}(\omega^{-1})$, par la proposition $2.1$ de \cite{Schoof} on a $\frac{-4}{3} \cdot \sum_{k=1}^{\frac{N-1}{2}} k\cdot \text{log}(k)^i \equiv 0 \text{ (mod }p\text{)}$, d'où le résultat par récurrence.

Il serait intéressant de prouver que la réciproque du théorème~\ref{r_0>i} est vraie pour tout $1 \leq i \leq p-1$. En utilisant les représentations galoisiennes construites par Calegari--Emerton et la théorie des symboles modulaires, on montrera dans un prochain article que la réciproque est vraie pour $i \leq 2$, au moins si $\nu = 1$.

\section{Lien avec la conjecture de Gross}
Soit $\chi : (\mathbf{Z}/p\mathbf{Z})^{\times} \rightarrow \mathbf{Z}_p^{\times}$ un caractère différent de $\omega$ (donc $\chi = \omega^i$ pour $i$ entier, $i \not\equiv 1 \text{ (mod }p\text{)}$). Soit $H = (\mathbf{Z}/N\mathbf{Z})^{\times}$. Rappelons que $I_H \subset \mathbf{Z}_p[H]$ est l'idéal d'augmentation.

Soit $\overline{B}_1 : \mathbf{R} \rightarrow \mathbf{R}$ le premier polynôme de Bernoulli périodique défini par $\overline{B}_1(x) = x-[x]-\frac{1}{2}$ si $x \not\in \mathbf{Z}$ et  $\overline{B}_1(x) =0$ si $x\in \mathbf{Z}$.

On définit l'\textit{élément de Stickelberger tordu} par $\chi$, noté $\phi_{\chi} \in \mathbf{Q}_p[H]$ donné par :
$$\phi_{\chi} = \sum_{a \in (\mathbf{Z}/Np\mathbf{Z})^{\times}} \overline{B}_1(\frac{a}{Np})\cdot \chi^{-1}(a) \cdot [a^{-1}] \text{ .}$$

Il est classique que si $\chi \neq 1, \omega$, $\phi_{\chi} \in \mathbf{Z}_p[H]$ (nous le démontrons  dans la proposition ci-dessous).

\begin{prop}\label{Stickelberger}
Si $\chi \neq 1,\omega$, on a $\phi_{\chi} \in I_H$. On a un homomorphisme de groupes $ \mathcal{L} : I_H/I_H^2 \rightarrow \mathbf{Z}/p^{\nu}\mathbf{Z}$ donnée par $[a]-1 \mapsto  \log(a)$. On a :
$$ \mathcal{L}(\phi_{\chi}) \equiv -\sum_{r=1}^{N-1} (\sum_{a=1}^{r-1} \chi^{-1}(a))\cdot  \log(r) \text{ (mod }p^{\nu} \text{)}\text{ .}$$ En particulier si $\chi=\omega^{-1}$, cette image modulo $p$ est $$\frac{2}{3}\cdot \sum_{k=1}^{\frac{N-1}{2}} k\cdot  \log(k) \text{ (mod }p\text{)} \text{ .}$$
\end{prop}
\begin{proof}
Soit $r$ un entier premier à $N$ tel que $0 < r < N$.
Le coefficient $\alpha_r$ en $[r^{-1}]$ de $\phi_{\chi}$ est 
$$\alpha_r = \sum_{a\in G, a\equiv r \text{ (mod }N\text{)}} \overline{B}_1(\frac{a}{Np})\cdot \chi^{-1}(a) = \sum_{k=0, k \neq p-r}^{p-1} \overline{B}_1(\frac{kN+r}{Np})\cdot \chi^{-1}(kN+r) \text{ .}$$
Comme $\chi \neq 1$, 
$$\alpha_r = \sum_{k=0, k \neq p-r}^{p-1} \frac{k}{p}\cdot \chi^{-1}(k+r) = \frac{1}{p}\cdot \left( \sum_{a=1}^{p-1}a\cdot \chi^{-1}(a) + p\cdot \sum_{a=1}^{r-1}\chi^{-1}(a) \right) \text{ .}$$
Donc 
$$\alpha_r = B_{1, \chi^{-1}} + \sum_{a=1}^{r-1}\chi^{-1}(a) \text{ .}$$
où $B_{1, \chi^{-1}} = \sum_{a \in (\mathbf{Z}/p\mathbf{Z})^{\times}} \overline{B}_1(\frac{a}{p})\cdot \chi^{-1}(a)$ est le premier nombre de Bernoulli généralisé associé à $\chi^{-1}$.
D'après \cite[Corollary $5.15$]{Washington} (comme $\chi \neq \omega$), $B_{1,\chi} \in \mathbf{Z}_p$ (et est congru à $\frac{B_{i+1}}{i+1}$ modulo $p$ si $\chi = \omega^i$). Donc $\alpha_r \in \mathbf{Z}_p^{\times}$, et on a la formule pour $ \mathcal{L}(\phi_{\chi})$. La dernière assertion vient du :

\begin{lem}\label{B_2}
On a :
$$\sum_{k=1}^{N-1} k^2\cdot  \log(k) \equiv \frac{-4}{3}\cdot \left( \sum_{k=1}^{\frac{N-1}{2}}k\cdot  \log(k)\right) \text{ (mod }p^{\nu}\text{)} \text{ .}$$
\end{lem}
\begin{proof}
On pose $V = \sum_{t=1}^{\frac{N-1}{2}} t\cdot  \log(t)$. Soit $\overline{B}_2 : \mathbf{R} \rightarrow \mathbf{R}$ défini par 
$$\overline{B_2}(x) = B_2(x-[x])$$
où $B_2(x) = x^2-x+\frac{1}{6}$ est le deuxième polynôme de Bernoulli et $[x]$ est la partie entière (inférieure) de $x$.
On a alors $$V = \sum_{t \in (\mathbf{Z}/N\mathbf{Z})^{\times}} F(\frac{t}{N})\cdot  \log(t)$$ où $$F(x) = \frac{1}{4}\overline{B_2}(2x) - \overline{B_2}(x) + \frac{1}{2}\overline{B_1}(x-\frac{1}{2}) + \frac{1}{8} \text{ .}$$
Par conséquent, on a $$ V = \frac{-3}{4} \cdot \sum_{t \in (\mathbf{Z}/N\mathbf{Z})^{\times}} \overline{B}_2(\frac{t}{N})\cdot  \log(t)$$ donc
\begin{align*}
V &=  \frac{-3}{4} \cdot \sum_{t=1}^{N-1} t^2\cdot  \log(t)
\end{align*}
\end{proof}

\end{proof}

Soit $\Gamma_N : \mathbf{Z}_N \rightarrow \mathbf{Z}_N^{\times}$ la fonction gamma $N$-adique de Morita. C'est l'unique fonction continue $\mathbf{Z}_N \rightarrow \mathbf{Z}_N^{\times}$ vérifiant $\Gamma_N(n) = (-1)^n\cdot \prod_{1 \leq i \leq n-1, \text{ pgcd(}n,N \text{)}=1} i$ si $n>1$ est un entier.
Rappelons quelques propriétés de $\Gamma_N$ (qui caractérisent en fait $\Gamma_N$, par \cite[p.$524$]{Hayes}). 
\begin{itemize}
\item[$\bullet$]$\Gamma_N(0)=1$.
\item[$\bullet$]Si $z\in \mathbf{Z}_N^{\times}$, $\Gamma_N(1-z) = z\cdot \Gamma_N(-z)$.
\item[$\bullet$]Si $z \in N\mathbf{Z}_N$, $\Gamma_N(1-z) = -\Gamma_N(z)$.
\end{itemize}

Si $a \in \mathbf{Q}_N^{\times}$ et $t \in \mathbf{Z}_p$, notons 
$$a^t := a\otimes t \in \mathbf{Q}_N^{\times} \otimes_{\mathbf{Z}} \mathbf{Z}_p \text{ .}$$
Rappelons qu'on a défini un logarithme discret
$$\Lambda: \mathbf{Q}_N^{\times} \otimes_{\mathbf{Z}} \mathbf{Z}_p \rightarrow \mathbf{Z}/p^{\nu}\mathbf{Z} \text{ .}$$ 

\begin{prop}\label{Gamma}
Soit $\chi : (\mathbf{Z}/p\mathbf{Z})^{\times} \rightarrow \mathbf{Z}_p$ un caractère $\neq 1, \omega$. Alors on a 
$$\Lambda \left( \prod_{a=1}^{p-1}\Gamma_N(\frac{a}{p})^{\chi^{-1}(a)} \right) \equiv \chi(-1)\cdot  \mathcal{L}(\phi_{\chi}) \equiv -\chi(-1)\cdot \sum_{r=1}^{N-1}\cdot (\sum_{a=1}^{r-1} \chi^{-1}(a))\cdot  \log(r)  \text{ (mod }p^{\nu}\text{)} $$
(\textit{cf.} proposition~\ref{Stickelberger}). En particulier si $\chi = \omega^{-1}$, on obtient la congruence suivante (on voit $\log$ comme un morphisme $\mathbf{Z}_N^{\times} \rightarrow \mathbf{F}_p$ via la projection $\mathbf{Z}_N^{\times} \rightarrow \mathbf{F}_N^{\times}$) :
$$\sum_{a=1}^{p-1} a\cdot  \log(\Gamma_N(\frac{a}{p})) \equiv \frac{-2}{3} \cdot \sum_{k=1}^{\frac{N-1}{2}} k\cdot  \log(k) \text{ (mod }p\text{)} \text{ .}$$
\end{prop}

\begin{proof}
Remarquons d'abord que si $z \in N\mathbf{Z}_N$, on a $\Gamma_N(z) \equiv \pm 1 \text{ (mod }N\text{)}$.
En effet c'est vrai si $z \in \mathbf{N}$ car $(N-1)! \equiv -1 \text{ (mod }N\text{)}$, or $\mathbf{N}$ est dense dans $\mathbf{Z}_p$ et $\Gamma_N$ est continue.

On suit \cite[$p.525-526$]{Hayes}. Par l'équation fonctionnelle satisfaite par $\Gamma_N$, pour tous entiers $e$ et $n$ avec $0 \leq e\leq n$ on a :
$$\Gamma_N(1-z-e) = \left( \prod_{r=e}^n \{z+r\} \right)\cdot \Gamma_N(-z-n)$$ où $\{z\}=z$ si $z \in \mathbf{Z}_N^{\times}$ et $\{z\}=-1$ si $z\in N\mathbf{Z}_N$.

\begin{lem}
Soit $1 \leq a \leq p-1$ un entier. On a : 
$$\Gamma_N(1-\frac{a}{p}) \equiv \pm (\frac{a}{p})^{a\cdot \frac{N-1}{p}}\prod_{x \in X} x^{[\frac{ax}{Np}]+[\frac{a(N-x)}{Np}]+1}  \text{ (mod }N\text{)} $$
 où $X$ est n'importe quel ensemble d'entiers représentant les classes inversibles modulo $N$, avec la condition que tout $x \in X$ soit $\equiv 1 \text{ (mod }p\text{)}$ ($[.]$ est la partie entière inférieure).
\end{lem}
\begin{proof}
On choisit $z = \frac{a}{p}$ avec $1 \leq a \leq p-1$, $e=0$ et $n=a\cdot\frac{N-1}{p}$ dans l'identité précédent le lemme.
Cela donne :
$\Gamma_N(1-\frac{a}{p}) = \prod_{r=0}^{a\cdot \frac{N-1}{p}} \{\frac{a}{p}+r\}\cdot \Gamma_N(\frac{-aN}{p})$. On a vu au début de la démonstration de la proposition que $\Gamma_N(\frac{-aN}{p}) \equiv \pm 1  \text{ (mod }N\text{)}$. 
Cela montre :  
$$\Gamma_N(1-\frac{a}{p}) = \pm \prod_{r=0}^{a\cdot \frac{N-1}{p}} \{\frac{a}{p}+r\} \text{ .}$$

Pour tout $0 \leq r  < a\cdot \frac{N-1}{p}$, on a $\frac{a}{p}+r \in \mathbf{Z}_N^{\times}$. En effet, si $\frac{a}{p}+r = \frac{k}{p}\cdot N$ pour un entier $k$, on a $k >0$ et $r = \frac{kN-a}{p} \in \mathbf{N}$, donc $k \equiv a  \text{ (mod }p\text{)}$, i.e. $k = k'p+a$ avec $k' \geq 0$, donc $r \geq a \cdot \frac{N-1}{p}$.

Soit $X$ un ensemble comme dans l'énoncé du lemme. Soit $x \in X$.  On considère les entiers $r$ tels que $1+\frac{pr}{a}\equiv x \text{ (mod }N\text{)}$ et $0 \leq r \leq a\cdot \frac{N-1}{p}$. \\ On écrit $1+\frac{pr}{a} = x +\frac{p\cdot t}{a}\cdot N$ pour un $t \in \mathbf{Q}$. On a $t = \frac{a+pr-ax}{Np}$. On a $a+pr-ax \equiv 0 \text{ (mod }p\text{)}$ car $x \equiv 1 \text{ (mod }p\text{)}$ et   $a+pr-ax \equiv 0 \text{ (mod }N\text{)}$ car  $1+\frac{pr}{a}\equiv x \text{ (mod }N\text{)}$. On a donc $t \in \mathbf{Z}$. Réciproquement, si $t \in \mathbf{Z}$, on définit $r$ par $1+\frac{pr}{a} = x +\frac{p\cdot t}{a}\cdot N$, c'est à dire $r =\frac{a(x-1)}{p}+t\cdot N$, et on a $r \in \mathbf{Z}$ car $x \equiv 1 \text{ (mod }p\text{)}$. En écrivant $-\frac{a}{p} < r \leq \frac{a(N-1)}{p}$, on obtient $\frac{-ax}{Np} < t \leq \frac{a(N-x)}{Np}$, donc il y a $[\frac{a(N-x)}{Np}]-[\frac{-ax}{Np}] = [\frac{a(N-x)}{Np}]+[\frac{ax}{Np}]+1$ possibilités pour $t$, ce qui montre le lemme.

\end{proof}

Fixons dans toute la suite l'ensemble $X = \{N-k\cdot p, k\in \{1,2,..., N-1\}\}$. On peut aussi écrire $X = \{1+k\cdot p, k\in  \mathcal{K}\}$ où $\mathcal{K} = \{\frac{N-1}{p}-(N-1), ...,\frac{N-1}{p}-1\}$.

On a :
$$\prod_{x \in X} x \equiv (N-1)! \equiv -1  \text{ (mod }N\text{)} \text{ .}$$
Donc :
$$\Gamma_N(1-\frac{a}{p}) = \pm (\frac{a}{p})^{a\cdot \frac{N-1}{p}} \cdot \prod_{x\in X} x^{[\frac{ax}{Np}]+[\frac{a(N-x)}{Np}]} \text{ .}$$
Fixons dans toute la suite l'ensemble $X = \{N-k\cdot p, k\in \{1,2,..., N-1\}\}$.
On a : 
$$\prod_{x \in X} x^{[\frac{a(N-x)}{Np}]} \equiv \prod_{y \in Y} (-y)^{[\frac{ay}{Np}]}  \text{ (mod }N \text{)} $$
où $Y = \{k \cdot p, k \in \{1,2,...,N-1\}\}$.

On a : 
$$\prod_{y \in Y} (-y)^{[\frac{ay}{Np}]} \equiv \prod_{k=1}^{N-1}(-kp)^{[\frac{ak}{N}]} \text{ (mod }N\text{)} \text{ .}$$

\begin{lem}\label{rectangle}
On a, pour tout $a\geq 1$ entier : 
$$\sum_{k=1}^{N-1} [\frac{ak}{N}] = \frac{(a-1)(N-1)}{2} \text{ .}$$
\end{lem}
\begin{proof}
Notons $G$ la quantité de gauche. Alors $2(G + [\frac{aN}{N}] + N) = 2(G+a+N)$ est le nombre de points à coordonnées entières (bords compris) dans un rectangle à sommets entiers et de côtés de longueurs $a$ et $N$, c'est à dire $(a+1)(N+1)$. Donc $G = \frac{(a-1)(N-1)}{2}$.
\end{proof}

\begin{lem}
On a $$\left( \prod_{k=1}^{N-1} k^{[\frac{ak}{N}]} \right)^4 \equiv 1 \text{ (mod }N\text{)} \text{ .}$$
\end{lem}
\begin{proof}

En effet, on regroupe $k$ et $N-k$, en utilisant le fait que $[\frac{a(N-k)}{N}] = a -[\frac{ak}{N}]-1$ et que $(\frac{N-1}{2})!^4 \equiv 1  \text{ (mod }N\text{)}$.
\end{proof}

En utilisant les deux lemmes précédents, on obtient 
$$\chi^{-1}(-1)\cdot \Lambda \left( \prod_{a=1}^{p-1}\Gamma_N(\frac{a}{p})^{\chi^{-1}(a)} \right)  =  \Lambda \left (\prod_{a=1}^{p-1} (\frac{a}{p})^{a\cdot \chi^{-1}(a)\cdot \frac{N-1}{p}} \right) + \Lambda \left( \prod_{a=1}^{p-1} \prod_{x \in X} x^{[\frac{ax}{Np}]\cdot \chi^{-1}(a)} \right) \text{ .}$$

Pour finir la preuve du théorème~\ref{Gamma}, il s'agit de montrer l'identité :
$$\sum_{a=1}^{p-1}\sum_{x\in X} \chi^{-1}(a)\cdot(a\cdot\frac{N-1}{p})\cdot  \log(a) + [\frac{ax}{Np}]\cdot \chi^{-1}(a) \cdot  \log(x) \equiv  \mathcal{L}(\phi_{\chi}) \text{ (mod }p^{\nu}\text{)} \text{ .}$$

Rappelons que $\phi_{\chi} = \sum_{a=1}^{p-1}\sum_{x \in X} \overline{B}_1(\frac{ax}{Np})\cdot \chi^{-1}(a)\cdot [(ax)^{-1}]$, et qu'en fait $\phi_{\chi} \in \mathbf{Z}_p[(\mathbf{Z}/N\mathbf{Z})^{\times}]$. Soit $ \lambda : (\mathbf{Z}/N\mathbf{Z})^{\times} \rightarrow \mathbf{Z}$ un relèvement du $ \log : (\mathbf{Z}/N\mathbf{Z})^{\times} \rightarrow  \mathbf{Z}/p^{\nu}\mathbf{Z}$ fixé dans cet article. 
Soit $$S = \sum_{a=1}^{p-1}\sum_{x \in X} \chi^{-1}(a)(\frac{ax}{Np}-[\frac{ax}{Np}])\cdot  \lambda(ax)\in \mathbf{Z}_p $$
(le fait que $S \in \mathbf{Z}_p$ vient du fait que $\phi_{\chi} \in \mathbf{Z}_p[(\mathbf{Z}/N\mathbf{Z})^{\times}]$.)
Alors on a :
$$ \mathcal{L}(\phi_{\chi}) \equiv -S \text{ (mod }p^{\nu}\text{)}\text{ .}$$
On a :
$$S = S_1-S_2$$
où $$S_1 = \sum_{a=1}^{p-1}\sum_{x \in X} \chi^{-1}(a)\cdot\frac{ax}{Np}\cdot  \lambda(ax)$$
et
$$S_2 = \sum_{a=1}^{p-1}\sum_{x \in X} \chi^{-1}(a)\cdot[\frac{ax}{Np}]\cdot  \lambda(ax) \text{ .}$$
On a $S_1$ et $S_2$ $\in \mathbf{Z}_p$. De plus 
$$S_2 \equiv \left( \sum_{a=1}^{p-1} \chi^{-1}(a) \cdot a\cdot \frac{N-1}{p} \cdot \log(a)\right) +  \left( \sum_{a=1}^{p-1} \sum_{x \in X} \chi^{-1}(a)\cdot [\frac{ax}{Np}] \cdot \log(x) \right) \text{ (mod }p^{\nu}\text{)} \text{ .}$$
par l'identité suivante :
\begin{lem}
On a, pour $a$ entier, $1 \leq a \leq p-1$,
$$\sum_{x \in X} [\frac{ax}{Np}] \equiv a\cdot \frac{N-1}{p}  \text{ (mod }p^{\nu}\text{)} \text{ .}$$
\end{lem}
\begin{proof}
Rappelons qu'on a  $X = \{kp+1, k \in \mathcal{K} \}$ où $\mathcal{K} = \{\frac{N-1}{p}-(N-1), ...,\frac{N-1}{p}-1\}$.
Si $x = k\cdot p+1 \in X$, on a $[\frac{a\cdot x}{Np}] = [\frac{ak}{N}+\frac{a}{Np}] = [\frac{ak}{N}]$.
Donc :
\begin{align*}
\sum_{x \in X} [\frac{ax}{Np}] &= \sum_{k=\frac{N-1}{p}-(N-1)}^{\frac{N-1}{p}-1} [\frac{ak}{N}] = \left( \sum_{k=1}^{\frac{N-1}{p}-1} [\frac{ak}{N}] \right) - \left( \sum_{k=1}^{N-1-\frac{N-1}{p}}[\frac{ak}{N}]\right) - (N-1-\frac{N-1}{p}) \\& = \left( \sum_{k=N-\frac{N-1}{p}}^{N-1} [\frac{a(N-k)}{N}] \right)\ - \left( \sum_{k=1}^{N-1-\frac{N-1}{p}}[\frac{ak}{N}]\right) - (N-1-\frac{N-1}{p}) \\& = \frac{(a-1)(N-1)}{p}-\sum_{k=1}^{N-1} [\frac{ak}{N}] - (N-1-\frac{N-1}{p}) \text{ .}
\end{align*}
D'après le lemme~\ref{rectangle}, $\sum_{k=1}^{N-1} [\frac{ak}{N}] = \frac{(a-1)(N-1)}{2}$, d'où :
\begin{align*}
\sum_{x \in X} [\frac{ax}{Np}] &= \frac{(a-1)(N-1)}{p} - \frac{(a-1)(N-1)}{2} - (N-1-\frac{N-1}{p}) \equiv a \cdot \frac{N-1}{p}  \text{ (mod }p^{\nu}\text{)}
\end{align*}
ce qui achève la démonstration du lemme.
\end{proof}
Pour finir la démonstration de la proposition~\ref{Gamma}, il reste à montrer que $S_1 \equiv 0 \text{ (mod }p^{\nu}\text{)}$.
On a :
$$S_1 = \frac{1}{Np}\cdot \left( \sum_{a=1}^{p-1}\chi^{-1}(a)\cdot a\cdot (\sum_{x\in X} x\cdot  \lambda(ax) )\right) \text{ .}$$
On a : 
$$\sum_{x \in X} x\cdot  \lambda(ax) = \sum_{k \in \mathcal{K}} (1+kp)\cdot  \lambda(ax) = \left( \sum_{x \in X}  \lambda(ax) \right)+p\cdot \left(\sum_{k \in \mathcal{K}} k\cdot  \lambda((kp+1)a)\right)\text{ .}$$
Le terme
$$\sum_{x \in X}  \lambda(ax) $$
ne dépend pas de $a$ et est divisible par $p^{\nu}$. Comme $\frac{1}{p}\sum_{a=1}^{p-1} a\cdot \chi^{-1}(a) = B_{1,\chi^{-1}} \in \mathbf{Z}_p$, on a :
$$S_1 \equiv \frac{1}{N}\cdot \left(\sum_{a=1}^{p-1} \chi^{-1}(a)\cdot a \cdot (\sum_{k \in \mathcal{K}} k\cdot  \log((kp+1)a)) \right) \text{ (mod }p^{\nu}\text{)} \text{ .}$$
Donc :
$$S_1 \equiv \left((\sum_{k \in \mathcal{K}} k) \cdot \sum_{a=1}^{p-1}  \chi^{-1}(a)\cdot a \cdot  \log(a) \right) + \left( \sum_{a=1}^{p-1} \chi^{-1}(a)\cdot a \cdot (\sum_{k \in \mathcal{K}}  k \cdot \log(kp+1))\right) \text{ (mod }p^{\nu}\text{)} \text{ .}$$
On a 
$$\sum_{k \in \mathcal{K}} k = (N-1)\cdot \frac{N-1}{p} - \frac{N(N-1)}{2} \equiv 0  \text{ (mod }p^{\nu}\text{)} \text{ .}$$
On a 
$$  \sum_{a=1}^{p-1} \chi^{-1}(a)\cdot a \cdot (\sum_{k \in \mathcal{K}}  k \cdot \log(kp+1))  \equiv B_{1, \chi^{-1}}\cdot \sum_{k \in \mathcal{K}} k\cdot p \cdot \log(kp+1) \equiv B_{1, \chi^{-1}}\cdot \sum_{x \in X} x\cdot \log(x) \text{ (mod }p^{\nu}\text{)} $$
(la dernière congruence vient du fait que $\sum_{x \in X} \log (x) \equiv 0  \text{ (mod }p^{\nu}\text{)}$).
On a : 
$$\sum_{x \in X} x\cdot \log(x) \equiv \sum_{k=1}^{N-1} (N-kp)\cdot \log(kp) \equiv 0 \text{ (mod }p^{\nu}\text{)}$$
car $\sum_{k=1}^{N-1} k\cdot \log(k) \equiv  \sum_{k=1}^{\frac{N-1}{2}} \log(k) \equiv \log((\frac{N-1}{2}) !) \equiv 0\text{ (mod }p^{\nu}\text{)}$ (on regroupe les termes en $k$ et $N-k$ pour la première congruence).

On obtient finalement : 
$$S_1 \equiv 0 \text{ (mod }p^{\nu}\text{)}\text{ .}$$
La preuve de la proposition~\ref{Gamma} est donc achevée.
\end{proof}

\section{Démonstration par la formule de Gross--Koblitz}

Achevons cet article en montrant les théorèmes~\ref{identite_u},~\ref{identite_garett} et~\ref{P}.

Rappelons qu'on a fixé un idéal premier $\mathfrak{p}$ au-dessus de $(N)$ dans $\mathbf{Q}(\zeta_p)$. Soit $\mathfrak{P}$ l'idéal premier au-dessus de $\mathfrak{p}$ dans $F:=\mathbf{Q}(\zeta_p, \zeta_N)$. On a un plongement canonique $F \hookrightarrow F_{\mathfrak{P}}$ où $F_{\mathfrak{P}}$ est le complété $\mathfrak{P}$-adique de $F$. Rappelons qu'on a défini une somme de Gauss $\mathcal{G}$.

\begin{thm}(Formule de Gross-Koblitz, \cite[Theorem $1.7$]{Gross_Koblitz})
Soit $\pi \in F_{\mathfrak{P}}$ l'unique uniformisante vérifiant :
\begin{itemize}
\item[$\bullet$]$\pi^{N-1}=-N$,
\item[$\bullet$]$\pi \equiv \zeta_N-1 \text{ (mod }(\zeta_N-1)^2\text{)}$.
\end{itemize}

Alors dans $K_{\mathfrak{P}}$ on a l'égalité suivante pour tout $a\in \{1,...,p-1\}$ :
$$\sigma_a(\mathcal{G}) = \pi^{a\cdot \frac{N-1}{p}} \cdot \Gamma_N(\frac{a}{p}) \text{ .}$$
\end{thm}

Par la proposition~\ref{unit_gauss} et la formule de Gross--Koblitz ci-dessus, on a $$u_{\omega^{-i}}^{B_{1, \omega^i}} = \pi^{(N-1)\cdot B_{1,\omega^i}}\cdot \prod_{a=1}^{p-1} \Gamma_N(\frac{a}{p})^{\omega^i(a)} = (-N)^{B_{1,\omega^i}}\cdot \prod_{a=1}^{p-1} \Gamma_N(\frac{a}{p})^{\omega^i(a)} \in \mathbf{Q}_N^{\times} \otimes_{\mathbf{Z}} \mathbf{Z}_p \text{ .}$$
Cette égalité, combinée avec la proposition~\ref{Gamma}, prouve le théorème~\ref{identite_u}.

Si $\mathfrak{p}$ est principal, engendré par $\epsilon \in \mathbf{Z}[\zeta_p]$, on a par construction :

$$u_{\omega^{-i}}^{B_{1, \omega^i}} = \epsilon^{B_{1, \omega^i}} \cdot \prod_{a \in \mathbf{F}_p^{\times}\text{, } a\not\equiv 1} \log(\sigma_a(\epsilon))^{\omega^i(a)\cdot B_{1,\omega^i}}$$
et
$$N = \prod_{a \in \mathbf{F}_p^{\times}} \sigma_a(\epsilon)$$
donc
$$\prod_{a \in \mathbf{F}_p^{\times}\text{, } a\not\equiv 1} \sigma_a(\epsilon)^{(\omega^i(a)-1)\cdot B_{1, \omega^i}} = (-1)^{B_{1,\omega^i}}\cdot \prod_{a=1}^{p-1}\Gamma_N(\frac{a}{p})^{\omega^i(a)} \text{ .}$$
D'après la proposition~\ref{Gamma}, on obtient la formule annoncée dans le théorème~\ref{P}.

Pour conclure cet article, il reste à montrer le théorème~\ref{identite_garett}.

On doit montrer que $\Lambda(\frac{\beta}{N})=0 \text{ (mod }p\text{)}$ où $\beta \in \mathbf{Q}(\zeta_p)^{\times}$ est tel que $K_0 = \mathbf{Q}(\zeta_p)(\beta^{\frac{1}{p}})$ (on voit $\beta$ comme un élément de $\mathbf{Q}_N^{\times}$ via la complétion de $\mathbf{Q}(\zeta_p)$ en $\mathfrak{p}$).  Par la théorie de Galois, on a 
$$\mathcal{G}^p \in \mathbf{Z}[\zeta_p] \text{ .}$$
Comme $\mathcal{G} \in \mathbf{Q}(\zeta_p, \zeta_N)$, on en déduit qu'on peut choisir $\beta = \mathcal{G}^p$.
Les congruences de Kummer (\textit{cf.} \cite[p.5]{garett}) nous donnent :
$$\frac{\mathcal{G}}{(\zeta_N-1)^{\frac{N-1}{p}}} \equiv \frac{-1}{(\frac{N-1}{p})!}  \text{ (mod }\mathfrak{P}\text{)} \text{ .}$$
On a :
$$N=(\zeta_N-1)\cdot(\zeta_N^2-1)\cdot ... \cdot (\zeta_N^{N-1}-1)$$
et comme $\zeta_N \equiv 1 \text{ (mod }\mathfrak{P}\text{)}$, on obtient 
$$\frac{N}{(\zeta_N-1)^{N-1}} \equiv (N-1)! \equiv -1 \text{ (mod }\mathfrak{P}\text{)} \text{ .}$$
Donc :
$$\frac{\beta}{N}=\frac{\mathcal{G}^p}{N} \equiv -(\frac{-1}{(\frac{N-1}{p})!})^p \text{ (mod }\mathfrak{p}\text{)} \text{ .}$$
Cela montre que $\frac{\beta}{N}$ est une puissance $p$-ième modulo $\mathfrak{p}$, donc que $\Lambda(\frac{\beta}{N})=1$.

Cela achève la preuve du théorème~\ref{identite_garett}.

\bibliography{ClassGroups}
\bibliographystyle{alpha}
\newpage

\end{document}